\documentclass[12pt,svgnames]{article}
\usepackage{amsmath}
\usepackage{amssymb}
\usepackage{amsthm}
\usepackage{wasysym}
\usepackage{enumerate}
\usepackage{dsfont}
\usepackage{stmaryrd}
\usepackage{comment}
\usepackage{tikz}
\usepackage{tikz-cd}
\usepackage{extarrows}
\usepackage{mathrsfs}
\usepackage{enumitem}
\usepackage{fullpage}
\usepackage{hyperref}
\usepackage{accents}
\usepackage{titlesec}
\usepackage{quiver}
\usepackage[all]{xy}
\hypersetup{
    linkbordercolor = {PaleTurquoise},
    citebordercolor = {PaleGreen},
}

\titleformat{\section}{\normalsize\bfseries}{\thesection}{1em}{}
\titleformat{\subsection}{\normalsize\bfseries}{\thesubsection}{1em}{}

\numberwithin{equation}{subsection}

\theoremstyle{definition}

\newtheorem{defn}[subsection]{Definition}
\newtheorem{para}[subsection]{}

\newtheorem*{assumption*}{Assumption}

\newtheorem{defn_sub}[subsubsection]{Definition}
\newtheorem{para_sub}[subsubsection]{}
\newtheorem{ex_sub}[subsubsection]{Example}
\newtheorem{rmk_sub}[subsubsection]{Remark}

\theoremstyle{plain}

\newtheorem{prop}[subsection]{Proposition}
\newtheorem{theo}[subsection]{Theorem}
\newtheorem{lem}[subsection]{Lemma}
\newtheorem{cor}[subsection]{Corollary}

\newtheorem*{claim*}{Claim}
\newtheorem*{just*}{Justification}
\newtheorem*{lem*}{Lemma}
\newtheorem*{prop*}{Proposition}

\newtheorem{prop_sub}[subsubsection]{Proposition}

\newcommand{\V}{\mathscr{V}}

\newcommand{\ob}{\mathop{\mathsf{ob}}}

\newcommand{\tensor}{\otimes}
\newcommand{\Mnd}{\mathsf{Mnd}}

\newcommand{\Cat}{\text{-}\mathsf{Cat}}

\newcommand{\underJ}{\kern -0.5ex \mathscr{J}}

\newcommand{\Set}{\mathsf{Set}}
\newcommand{\ev}{\mathsf{ev}}
\newcommand{\op}{\mathsf{op}}
\newcommand{\Th}{\mathsf{Th}}
\newcommand{\T}{\mathbb{T}}

\newcommand{\A}{\mathscr{A}}

\newcommand{\Mod}{\text{-}\mathsf{Mod}}

\newcommand{\scrS}{\mathscr{S}}

\newcommand{\Alg}{\text{-}\mathsf{Alg}}

\newcommand{\N}{\mathbb{N}}

\newcommand{\Monadic}{\text{-}\mathsf{Monadic}}

\newcommand{\Term}{\mathsf{Term}}

\newcommand{\sff}{\mathsf{sf}}
\newcommand{\scrT}{\mathscr{T}}
\newcommand{\Pow}{\mathsf{Pow}}

\newcommand{\CAT}{\text{-}\mathsf{CAT}}

\newcommand{\bbC}{\mathbb{C}}
\newcommand{\Gph}{\mathsf{Graph}}

\newcommand{\calE}{\mathcal{E}}

\newcommand{\Var}{\mathsf{Var}}
\newcommand{\Pos}{\mathsf{Pos}}

\newcommand{\Str}{\text{-}\mathsf{Str}}

\newcommand{\calT}{\mathcal{T}}

\newcommand{\sfE}{\mathsf{E}}

\newcommand{\calS}{\mathcal{S}}

\newcommand{\calC}{\mathcal{C}}

\newcommand{\calJ}{\mathcal{J}}
\newcommand{\calP}{\mathcal{P}}

\newcommand{\scrE}{\mathscr{E}}
\newcommand{\Top}{\mathsf{Top}}

\newcommand{\Preord}{\mathsf{Preord}}
\newcommand{\PMet}{\mathsf{PMet}}

\newcommand{\Simp}{\mathsf{Simp}}

\newcommand{\calQ}{\mathcal{Q}}
\newcommand{\QuasiSp}{\mathsf{QuasiSp}}

\newcommand{\vbar}{\vec{v}}
\newcommand{\scrP}{\mathscr{P}}
\newcommand{\VCAT}{\V\text{-}\mathsf{CAT}}
\newcommand{\frakT}{\mathfrak{T}}

\newcommand{\SFF}{\mathsf{SF}}

\newcommand{\RGph}{\text{-}\mathsf{RGph}}
\newcommand{\Met}{\mathsf{Met}}
\newcommand{\calR}{\mathcal{R}}
\newcommand{\CPO}{\mathsf{CPO}}
\newcommand{\DCPO}{\mathsf{DCPO}}
\newcommand{\oCPO}{\omega\text{-}\CPO}
\newcommand{\sfar}{\mathsf{ar}}
\newcommand{\Q}{\mathcal{Q}}
\newcommand{\UltMet}{\mathsf{UltMet}}

\begin{document}

\title{\Large \textbf{Strongly finitary monads and multi-sorted varieties enriched in cartesian closed concrete categories}}
\author{Jason Parker\let\thefootnote\relax\thanks{We acknowledge support from the Postdoctoral Program of the Department of Mathematics and Statistics at the University of Calgary, from the Natural Sciences and Engineering Research Council of Canada (NSERC), and from the Pacific Institute for the Mathematical Sciences (PIMS).}\medskip
\\
\small University of Calgary
\\
\small Calgary, Alberta, Canada}
\date{}

\maketitle

\begin{abstract}
It is a classical result of categorical algebra, due to Lawvere and Linton, that \emph{finitary varieties of algebras} (in the sense of Birkhoff) are dually equivalent to \emph{finitary monads} on $\Set$. Recent work of Ad\'{a}mek--Dost\'{a}l--Velebil has established that analogous results also hold in certain \emph{enriched} contexts. Specifically, taking $\V$ to be one of the cartesian closed categories $\mathsf{Pos}$, $\mathsf{UltMet}$, $\oCPO$, or $\DCPO$ of respectively posets, (extended) ultrametric spaces, $\omega$-cpos, or dcpos, Ad\'{a}mek--Dost\'{a}l--Velebil have shown that a suitable category of \emph{$\V$-enriched varieties of algebras} is dually equivalent to the category of \emph{strongly finitary $\V$-monads} on $\V$. 

In this paper, we extend and generalize these results in two ways: by allowing $\V$ to be an arbitrary complete and cocomplete cartesian closed category that is \emph{concrete} over $\Set$, and by also considering the multi-sorted case. Given a set $\calS$ of \emph{sorts}, we define a suitable notion of \emph{(finitary) $\V$-enriched $\calS$-sorted variety}, and we say that a $\V$-monad on the product $\V$-category $\V^\calS$ is \emph{strongly finitary} if its underlying $\V$-endofunctor is the left Kan extension of its restriction to a suitable full sub-$\V$-category $\N_\calS \hookrightarrow \V^\calS$. Our main result is that the category of $\V$-enriched $\calS$-sorted varieties is dually equivalent to the category of strongly finitary $\V$-monads on $\V^\calS$. By taking $\calS$ to be a singleton and $\V$ to be $\mathsf{Pos}, \mathsf{UltMet}, \oCPO$, or $\DCPO$, we thus recover the aforementioned results of Ad\'{a}mek--Dost\'{a}l--Velebil. We provide several classes of examples of $\V$-enriched $\calS$-sorted varieties, many of which admit very concrete, syntactic formulations.  
\end{abstract}

\section{Introduction}

It has been known since the work of Lawvere \cite{Law:PhD} and Linton \cite{Lintonequational} that \emph{finitary varieties of algebras} in the sense of Birkhoff \cite{Birkhoff}, i.e.~categories of finitary algebraic structures axiomatized by syntactic equations, are dually equivalent to \emph{finitary monads} on $\Set$, i.e.~monads whose underlying endofunctors preserve filtered colimits. It is also well known that a monad $\T = (T, \eta, \mu)$ on $\Set$ is finitary iff it is \emph{strongly finitary}, meaning that $T : \Set \to \Set$ is the left Kan extension of its restriction to the full subcategory $\mathsf{FinSet} \hookrightarrow \Set$ of finite sets. Thus, finitary varieties of algebras are dually equivalent to strongly finitary monads on $\Set$.

Recent work of Ad\'{a}mek--Dost\'{a}l--Velebil has shown that this classical result can be suitably generalized to certain \emph{enriched} settings. Given a complete and cocomplete cartesian closed category $\V$, we write $\N_\V \hookrightarrow \V$ for the full sub-$\V$-category consisting of the finite copowers $n \cdot 1$ ($n \in \N$) of the terminal object $1$. A $\V$-monad $\T = (T, \eta, \mu)$ on $\V$ is \emph{strongly finitary} \cite{KellyLackstronglyfinitary} if $T : \V \to \V$ is the left Kan extension of its restriction to $\N_\V \hookrightarrow \V$. Taking $\V$ to be the cartesian closed category $\Pos$ of posets, in \cite{categoricalviewordered} Ad\'{a}mek--Dost\'{a}l--Velebil established that the category of strongly finitary $\V$-monads on $\V = \Pos$ is dually equivalent to the category of \emph{varieties of ordered algebras}, which are $\Pos$-enriched categories of finitary algebras in $\Pos$ that satisfy certain syntactic \emph{inequations} between terms. Taking $\V$ to be the cartesian closed category $\oCPO$ of $\omega$-cpos, Ad\'{a}mek--Dost\'{a}l--Velebil subsequently proved in \cite{strongly_finitary_cts} that the category of strongly finitary $\V$-monads on $\V = \oCPO$ is dually equivalent to the category of \emph{varieties of continuous algebras}, which are ($\oCPO$)-enriched categories of finitary algebras in $\oCPO$ that satisfy certain syntactic equations between \emph{extended terms}; they also established a similar result for $\V = \DCPO$ (the cartesian closed category of dcpos). And taking $\V$ to be the cartesian closed category $\mathsf{UltMet}$ of (extended) \emph{ultrametric spaces}\footnote{An \emph{(extended) ultrametric space} is an (extended) metric space $\left(X, d : X \times X \to [0, \infty]\right)$ that satisfies the stronger triangle inequality $d(x, z) \leq \mathsf{max}\left\{d(x, y), d(y, z)\right\} (x, y, z \in X)$.}, in \cite{strongly_finitary_quant} Ad\'{a}mek--Dost\'{a}l--Velebil established that the category of strongly finitary $\V$-monads on $\V = \mathsf{UltMet}$ is dually equivalent to the category of \emph{varieties of ultra-quantitative algebras}, which are $\UltMet$-enriched categories of finitary algebras in $\mathsf{UltMet}$ that satisfy certain syntactic \emph{quantitative equations} between terms. 

In this paper, we show that these results of Ad\'{a}mek--Dost\'{a}l--Velebil extend to a more general enriched and \emph{multi-sorted} setting. Specifically, we consider an arbitrary complete and cocomplete cartesian closed category $\V$ that is \emph{concrete} over $\Set$, meaning that $\V$ is equipped with a faithful functor $|-| : \V \to \Set$ that is represented by the terminal object; in particular, the categories $\Pos, \UltMet, \oCPO$, and $\DCPO$ satisfy these conditions. Given such a category $\V$ and a set $\calS$ of \emph{sorts}, we define a notion of \emph{(finitary) $\V$-enriched $\calS$-sorted variety} and a notion of \emph{strongly finitary} $\V$-monad on the product $\V$-category $\V^\calS$. Our central result (Theorem \ref{variety_thm}) states that the category $\Var\left(\V^\calS\right)$ of $\V$-enriched $\calS$-sorted varieties is dually equivalent to the category $\Mnd_\sff\left(\V^\calS\right)$ of strongly finitary $\V$-monads on $\V^\calS$. As we show in \S\ref{examples_section}, when we take $\V$ to be $\Pos, \UltMet, \oCPO$, or $\DCPO$, our $\V$-enriched \emph{single-sorted} varieties recover the various notions of enriched variety studied by Ad\'{a}mek--Dost\'{a}l--Velebil in \cite{categoricalviewordered, strongly_finitary_cts, strongly_finitary_quant}, and thus our central result (Theorem \ref{variety_thm}) specializes to recover some of the central results of these papers.   

In detail, in \S\ref{theories_section} we first define a notion of \emph{$\V$-enriched $\calS$-sorted signature}, which is a set $\Sigma$ of \emph{operation symbols} equipped with an assignment to each operation symbol $\sigma \in \Sigma$ of a finite tuple $(S_1, \ldots, S_n) \in \calS^*$ of \emph{input sorts}\footnote{We shall actually define the notion of \emph{input sorts} in a  slightly different (but equivalent) way in Definition \ref{signature}, once we have established certain notation in \ref{given_data}.}, an \emph{output sort} $S \in \calS$, and a \emph{parameter object} $P$ of $\V$. A \emph{$\Sigma$-algebra} $A$ is then an object $A = \left(A_S\right)_{S \in \calS}$ of $\V^\calS$ equipped with, for each operation symbol $\sigma \in \Sigma$ as above, a $\V$-morphism
\[ \sigma^A : P \to \left[A_{S_1} \times \ldots \times A_{S_n}, A_S\right], \] where we write $[-, -]$ for the internal hom of the cartesian closed category $\V$. Under our hypotheses on $\V$, such a $\V$-morphism is equivalently given by a $|P|$-indexed family $\left(\sigma_p^A : A_{S_1} \times \ldots \times A_{S_n} \to A_S\right)_{p \in |P|}$ of $\V$-morphisms with the property that the function $\sigma^A : |P| \to \V\left(A_{S_1} \times \ldots \times A_{S_n}, A_S\right)$ given by $p \mapsto \sigma_p^A$ ($p \in |P|$) is \emph{$\V$-admissible}, i.e.~lifts (uniquely) to a $\V$-morphism $\sigma^A : P \to \left[A_{S_1} \times \ldots \times A_{S_n}, A_S\right]$. There is a $\V$-category $\Sigma\Alg$ of $\Sigma$-algebras, which is equipped with a forgetful $\V$-functor $U^\Sigma : \Sigma\Alg \to \V^\calS$ that sends each $\Sigma$-algebra $A$ to the object $A$ of $\V^\calS$.  

Every $\V$-enriched $\calS$-sorted signature $\Sigma$ canonically determines an associated \emph{classical} $\calS$-sorted signature $|\Sigma|$, and we then define a \emph{$\V$-enriched $\calS$-sorted equational theory} to be a pair $\calT = (\Sigma, \calE)$ consisting of a $\V$-enriched $\calS$-sorted signature $\Sigma$ and a set $\calE$ of \emph{syntactic $|\Sigma|$-equations} between terms over the associated classical $\calS$-sorted signature $|\Sigma|$. We write $\calT\Alg$ for the full sub-$\V$-category of $\Sigma\Alg$ consisting of the \emph{$\calT$-algebras}, i.e.~the $\Sigma$-algebras that satisfy all of the equations in $\calE$, and $\calT\Alg$ is equipped with a forgetful $\V$-functor $U^\calT : \calT\Alg \to \V^\calS$, so that we may regard $\calT\Alg$ as an object of the slice category $\V\CAT/\V^\calS$. We then define a \emph{$\V$-enriched $\calS$-sorted variety} to be an object of $\V\CAT/\V^\calS$ of the form $\calT\Alg$ for some $\V$-enriched $\calS$-sorted equational theory $\calT$. 

We write $\N_\calS \hookrightarrow \V^\calS$ for the full sub-$\V$-category of $\V^\calS$ consisting of the objects of the form $\left(n_S \cdot 1\right)_{S \in \calS}$, where each $n_S \in \N$ ($S \in \calS$) and $n_S = 0$ for all but finitely many $S \in \calS$; when $\calS$ is a singleton, we have $\N_\calS = \N_\V \hookrightarrow \V$. This full sub-$\V$-category is \emph{dense} (in the enriched sense), and thus it is a \emph{subcategory of arities} in the sense of \cite{Pres, Struct}. Adopting terminology introduced by Kelly and Lack \cite{KellyLackstronglyfinitary} in the single-sorted case (but noting that they did not assume that $\V$ is concrete over $\Set$), we say that a $\V$-monad $\T = (T, \eta, \mu)$ on $\V^\calS$ is \emph{strongly finitary} if $T : \V^\calS \to \V^\calS$ is the ($\V$-enriched) left Kan extension of its restriction to $\N_\calS \hookrightarrow \V^\calS$. 

In \S\ref{diag_section} we prove, using results of \cite{EP}, that every $\V$-enriched $\calS$-sorted equational theory \emph{presents} a strongly finitary $\V$-monad on $\V^\calS$. In \S\ref{converse_section}, again using results of \cite{EP}, we then establish the converse result, that every strongly finitary $\V$-monad on $\V^\calS$ is \emph{presented by} a $\V$-enriched $\calS$-sorted equational theory. These two results combine to yield the main result of the paper (Theorem \ref{variety_thm}), that the category $\Var\left(\V^\calS\right)$ of $\V$-enriched $\calS$-sorted varieties is dually equivalent to the category $\Mnd_\sff\left(\V^\calS\right)$ of strongly finitary $\V$-monads on $\V^\calS$. We conclude \S\ref{converse_section} with the result (Corollary \ref{variety_thm_2}) that the category $\Var\left(\V^\calS\right)$ is also dually equivalent to the category $\Th_{\N_\calS}\left(\V^\calS\right)$ of \emph{$\N_\calS$-theories} (in the sense of \cite{Struct}); when $\calS$ is a singleton, so that $\N_\calS = \N_\V \hookrightarrow \V$, these $\N_\V$-theories are precisely the enriched algebraic theories studied by Borceux and Day \cite{BorceuxDay}.

In \S\ref{examples_section}, we provide many examples of cartesian closed categories $\V$ that satisfy our assumptions, and in several cases we also establish even more concrete, syntactic formulations of $\V$-enriched $\calS$-sorted equational theories, in which arbitrary $\V$-enriched $\calS$-sorted signatures are replaced by \emph{classical} $\calS$-sorted signatures (where all parameter objects are trivial) together with certain additional syntactic constructions related to $\V$. We briefly discuss the classical example of $\V = \Set$ in \S\ref{set_subsection}, and how our main result (Theorem \ref{variety_thm}) specializes to recover the classical result that the category of finitary $\calS$-sorted varieties is dually equivalent to the category of (strongly) finitary monads on $\Set^\calS$ (see e.g.~\cite[Chapter 14]{Algebraic_theories}). In \S\ref{rel_subsection} we show that if $\T$ is any \emph{(reflexive) relational Horn theory} for which the category $\T\Mod$ of \emph{$\T$-models} is cartesian closed, then $\V = \T\Mod$ satisfies our assumptions. Prominent examples of $\V = \T\Mod$ include the categories $\Preord$ of preordered sets and $\Pos$ of partially ordered sets; the category $\calQ\Cat$ of small $\calQ$-categories for a complete Heyting algebra $\calQ$; the category $\UltMet$ of (extended) \emph{ultrametric spaces}; and every category of \emph{quasispaces} \cite{Dubucquasitopoi} on a \emph{finitary concrete site}, including the category of \emph{(abstract) simplicial complexes}. We show that ($\T\Mod$)-enriched $\calS$-sorted equational theories may be equivalently described in an even more concrete, syntactic manner as \emph{classical $\calS$-sorted equational theories with $\T$-relations}, which consist not only of syntactic equations between terms but also of syntactic \emph{relations} among terms. When $\T\Mod = \Pos$ (resp.~when $\T\Mod = \UltMet$), the ($\T\Mod$)-enriched categories of algebras of classical \emph{single-sorted} equational theories with $\T$-relations are precisely the varieties of ordered algebras of \cite{categoricalviewordered} (resp.~the varieties of ultra-quantitative algebras of \cite{strongly_finitary_quant}). In \S\ref{monotop_subsection}, we show that every cartesian closed category $\V$ whose representable functor $\V(1, -) : \V \to \Set$ is \emph{solid} (in the sense of \cite[Definition 25.10]{AHS}) satisfies our assumptions. In particular, every cartesian closed category $\V$ that is \emph{topological} or even \emph{monotopological} over $\Set$ satisfies our assumptions. When $\V$ is topological over $\Set$, the $\V$-enriched $\calS$-sorted equational theories studied in this paper coincide with those studied in \cite{Initial_algebras} (where $\V$ is assumed to be topological over $\Set$ but not necessarily cartesian closed). 

Finally, in \S\ref{CPO_subsection} we show that $\V = \oCPO$ and $\V = \DCPO$ satisfy our assumptions. Using the notion of \emph{$\omega$-cpo presentation} from \cite[\S 3]{Synth_top_prob} (adapted from the notion of \emph{dcpo presentation} developed in \cite{pres_dcpos}), we also show that ($\oCPO$)-enriched $\calS$-sorted equational theories may be equivalently described in an even more concrete, syntactic manner as \emph{classical $\calS$-sorted equational theories with ($\oCPO$)-relations}. In the single-sorted case, the latter theories are precisely those used to present the varieties of continuous algebras studied in \cite{strongly_finitary_cts}. 

The author wishes to acknowledge Rory Lucyshyn-Wright for an insightful discussion about the content of this paper.

\section{Background and given data}
\label{background}

We assume that the reader has (basic) familiarity with enriched category theory, including the (very) basic theory of weighted colimits, enriched Kan extensions, and enriched monads, as can be obtained from \cite{Dubucbook}, \cite{Kelly}, or \cite[Chapter 6]{Borceux2}; we mainly use the notation and terminology of \cite{Kelly}. 

Recall that a \emph{concrete category over $\Set$} is a category $\V$ equipped with a faithful functor $|-| : \V \to \Set$. We shall usually not distinguish notationally between a morphism $f : X \to Y$ of $\V$ and the underlying function $f = |f| : |X| \to |Y|$ in $\Set$. Given objects $X$ and $Y$ of $\V$ and a function $f : |X| \to |Y|$ in $\Set$, we say that $f$ is \emph{$\V$-admissible} if it lifts (necessarily uniquely) to a morphism $f : X \to Y$ of $\V$. Given an object $X$ of $\V$, we often refer to the set $|X|$ as the \emph{underlying set} of $X$.

\begin{para}[\textbf{Given data}]
\label{given_data}
For the remainder of the paper (unless otherwise stated), we fix the following data:
\begin{itemize}[leftmargin=*]
\item A category $\V$ with the following properties:
\begin{enumerate}
\item $\V$ is complete and cocomplete.
\item $\V$ is cartesian closed.
\item $\V$ is a concrete category over $\Set$, and the associated faithful functor $|-| : \V \to \Set$ is represented by the terminal object $1$ of $\V$, i.e.~$|-| \cong \V(1, -) : \V \to \Set$.
\end{enumerate}
We write the internal hom of the cartesian closed category $\V$ as $[-, -]$. Since $|-| \cong \V(1, -)$, we have $\left|[X, Y]\right| \cong \V(X, Y)$ (the hom-set in $\V$ from $X$ to $Y$) for all objects $X$ and $Y$ of $\V$ \cite[(1.25)]{Kelly}, and we shall assume without loss of generality that $\left|[X, Y]\right| = \V(X, Y)$. Also note that the representable functor $|-| : \V \to \Set$ preserves products, and we shall assume without loss of generality that $|-|$ \emph{strictly} preserves products, so that $|X \times Y| = |X| \times |Y|$ for all $X, Y \in \ob\V$ and $|1|$ is a singleton. Since $\V$ has coproducts, the functor $|-| : \V \to \Set$ has a left adjoint $F : \Set \to \V$ given by $S \mapsto S \cdot 1$ ($S \in \Set$). It then follows that the terminal object $1$ of $\V$ is \emph{discrete}, in the sense that every function $|1| = \{\ast\} \to |X|$ ($X \in \ob\V$) is $\V$-admissible. The discreteness of $1$ is equivalent to $\V$ \emph{admitting constant morphisms}, in the sense that every constant function $f : |X| \to |Y|$ ($X, Y \in \ob\V$) is $\V$-admissible.

The ordinary category $\V$ underlies a $\V$-category with hom-objects $[X, Y]$ ($X, Y \in \ob\V$), and we also denote this $\V$-category by $\V$. Since $1$ is discrete, we deduce from \cite[Proposition 27.18]{AHS} that for all objects $X$ and $Y$ of $\V$, the underlying function of the evaluation morphism $\varepsilon_{X, Y} : X \times [X, Y] \to Y$ (i.e.~the component at $Y$ of the counit of the adjunction $X \times (-) \vdash [X, -] : \V \to \V$) is given by the rule \[ (x, f) \in |X| \times \V(X, Y) \ \mapsto \ f(x) \in |Y|. \] 

We write $\V\CAT$ for the (non-locally-small) category of (possibly large) $\V$-categories and $\V$-functors. In \S\ref{examples_section}, we shall provide many examples of categories $\V$ that satisfy our main assumptions.   
 
\item A set $\calS$, whose elements we call \emph{sorts}. We write $\V^\calS$ for the product category $\prod_\calS \V$, which we may also regard as a $\V$-category, with hom-objects $\V^\calS(X, Y) = \prod_{S \in \calS} \left[X_S, Y_S\right]$ ($X, Y \in \ob\V^\calS$). Note that for all $X, Y \in \ob\V^\calS$, because the representable functor $|-| : \V \to \Set$ preserves products, the underlying set $\left|\V^\calS(X, Y)\right|$ of the hom-object $\V^\calS(X, Y)$ is the hom-set from $X$ to $Y$ in the ordinary category $\V^\calS$. We also write $\Set^\calS$ for the product category $\prod_\calS \Set$. 

We write $\N_\calS$ for the set of all functions $J : \calS \to \N$ with \emph{finite support}, meaning that $J(S) > 0$ for only finitely many $S \in \calS$. We shall see in \ref{N_S_para} that $\N_\calS$ can also be regarded as a (dense) full sub-$\V$-category of $\V^\calS$. 
\end{itemize}
\end{para}

\section{$\V$-enriched $\calS$-sorted equational theories}
\label{theories_section}

We begin by defining the notion of a $\V$-enriched $\calS$-sorted signature, which is adopted from \cite[Definition 3.1.1]{Initial_algebras} (see also \cite[Remark 3.1.3]{Initial_algebras} for related notions of signature, including \cite[Definition 3.1]{Battenfeld_comp} and the \emph{free-form $\N_\calS$-signatures} of \S\ref{diag_section} below).  

\begin{defn}
\label{signature}
A \textbf{$\V$-enriched $\calS$-sorted signature} is a set $\Sigma$ of \emph{operation symbols} equipped with an assignment to each operation symbol $\sigma \in \Sigma$ of \emph{input sorts} $J = J_\sigma \in \N_\calS$, an \emph{output sort} $S \in \calS$, and a \emph{parameter (object)} $P = P_\sigma \in \ob\V$, in which case we say that $\sigma$ has \emph{type} $(J, S, P)$. An operation symbol of $\Sigma$ is \textbf{ordinary} if its parameter object is the terminal object $1$ of $\V$. A \textbf{classical $\calS$-sorted signature} is a ($\V$-enriched) $\calS$-sorted signature whose operation symbols are all ordinary. When $\calS$ is a singleton, we refer to a $\V$-enriched $\calS$-sorted signature as a \textbf{$\V$-enriched single-sorted signature}.       
\end{defn}

\begin{para}[\textbf{The $\V$-functor $(-)^J$}]
\label{prod_not}
For each $J \in \N_\calS$, we write $(-)^J : \V^\calS \to \V$ for the $\V$-functor given by \[ X \mapsto X^J := \prod_{S \in \calS} X_S^{J(S)} \tag{$X \in \V^\calS$}. \] Thus, when $S_1, \ldots, S_m \in \calS$ are the finitely many sorts with $J\left(S_i\right) > 0$ ($1 \leq i \leq m$), we have 
\begin{equation}
\label{prod_not_eq}
X^J \cong X_{S_1}^{J\left(S_1\right)} \times \ldots \times X_{S_m}^{J\left(S_m\right)},
\end{equation}
$\V$-naturally in $X \in \V^\calS$. Once we have shown in \ref{N_S_para} that $\N_\calS$ may be regarded as a full sub-$\V$-category of $\V^\calS$, it will follow that the $\V$-functor $(-)^J : \V^\calS \to \V$ is represented by the object $J$ of $\N_\calS \hookrightarrow \V^\calS$.  
\end{para}

\begin{defn}
\label{sig_algebra}
Let $\Sigma$ be a $\V$-enriched $\calS$-sorted signature. A \textbf{$\Sigma$-algebra} $A$ is an object $A = \left(A_S\right)_{S \in \calS}$ of $\V^\calS$ equipped with, for each $\sigma \in \Sigma$ of type $(J, S, P)$, a $\V$-morphism \[ \sigma^A : P \to \left[A^J, A_S\right], \] which we may also equivalently write as
\[ \sigma^A : P \times A^J \to A_S. \]
When $\sigma$ is ordinary (so that $P = 1$), we simply write
\[ \sigma^A : A^J \to A_S. \] 
We often say that the object $A$ of $\V^\calS$ is the \emph{carrier (object)} of the $\Sigma$-algebra $A$. 

Given $\Sigma$-algebras $A$ and $B$, a \textbf{morphism of $\Sigma$-algebras} $f : A \to B$ is a morphism $f = \left(f_S\right)_{S \in \calS} : A = \left(A_S\right)_{S \in \calS} \to \left(B_S\right)_{S \in \calS} = B$ of $\V^\calS$ that makes either of the following equivalent diagrams commute for each $\sigma \in \Sigma$ of type $(J, S, P)$: 
\[\begin{tikzcd}
	{P \times A^J} &&& {P \times B^J} && P &&& {\left[B^J, B_S\right]} \\
	\\
	{A_S} &&& {B_S} && {\left[A^J, A_S\right]} &&& {\left[A^J, B_S\right]}.
	\arrow["{1 \times f^J}", from=1-1, to=1-4]
	\arrow["{\sigma^A}"', from=1-1, to=3-1]
	\arrow["{\sigma^B}", from=1-4, to=3-4]
	\arrow["{f_S}"', from=3-1, to=3-4]
	\arrow["{\sigma^A}"', from=1-6, to=3-6]
	\arrow["{\sigma^B}", from=1-6, to=1-9]
	\arrow["{\left[f^J, 1\right]}", from=1-9, to=3-9]
	\arrow["{\left[1, f_S\right]}"', from=3-6, to=3-9]
\end{tikzcd}\]
We have an ordinary category $\Sigma\Alg_0$ of $\Sigma$-algebras and their morphisms, which is equipped with a forgetful functor $U^\Sigma_0 : \Sigma\Alg_0 \to \V^\calS$ that sends each $\Sigma$-algebra $A$ to its carrier object.  

The category $\Sigma\Alg_0$ underlies a $\V$-category $\Sigma\Alg$ and the functor $U^\Sigma_0 : \Sigma\Alg \to \V^\calS$ underlies a $\V$-functor $U^\Sigma : \Sigma\Alg \to \V^\calS$. In detail, each hom-object $\Sigma\Alg(A, B)$ ($A, B \in \Sigma\Alg$) is defined as the \emph{pairwise equalizer} \cite[2.1]{Commutants} of the following $\Sigma$-indexed family of parallel pairs in $\V$ (one for each $\sigma \in \Sigma$ of type $\left(J, S, P\right)$):
\begin{equation}
\label{pp_eq}
\V^\calS(A, B) \xrightarrow{\left(P \times (-)^J\right)_{AB}} \left[P \times A^J, P \times B^J\right] \xrightarrow{\left[1, \sigma^B\right]} \left[P \times A^J, B_S\right],
\end{equation}
\begin{equation*}
\V^\calS(A, B) \xrightarrow{\left(\pi_S\right)_{AB}} \left[A_S, B_S\right] \xrightarrow{\left[\sigma^A, 1\right]} \left[P \times A^J, B_S\right],
\end{equation*}
where $\pi_S : \V^\calS \to \V$ is the projection $\V$-functor. Composition in $\Sigma\Alg$ is then defined in the unique way that allows the resulting monomorphisms $U^\Sigma_{AB} : \Sigma\Alg(A, B) \rightarrowtail \V^\calS(A, B)$ ($A, B \in \Sigma\Alg$) to be the structural morphisms of a faithful $\V$-functor $U^\Sigma : \Sigma\Alg \to \V^\calS$ that sends each $\Sigma$-algebra $A$ to its carrier object. Note that each structural morphism $U^\Sigma_{AB}$ ($A, B \in \Sigma\Alg$) is a regular and thus strong monomorphism. We can thus regard $\Sigma\Alg$ as a \emph{$\V$-category over $\V^\calS$}, i.e.~as an object of the slice category $\V\CAT/\V^\calS$.
\end{defn}

\noindent The following definition is essentially \cite[Definition 3.1.5]{Initial_algebras}.

\begin{defn}
\label{ord_sig}
Let $\Sigma$ be a $\V$-enriched $\calS$-sorted signature. The \textbf{underlying classical $\calS$-sorted signature (of $\Sigma$)} is the classical $\calS$-sorted signature $|\Sigma|$ \eqref{signature} defined as follows. Given an operation symbol $\sigma \in \Sigma$ of type $(J, S, P)$, we consider for each element $p \in |P|$ an ordinary operation symbol $\sigma_p$ with input sorts $J$ and output sort $S$. We then write $|\Sigma|$ for the classical $\calS$-sorted signature whose operation symbols are the ordinary operation symbols $\sigma_p$ for all $\sigma \in \Sigma$ and $p \in \left|P_\sigma\right|$ (with the types just described).         
\end{defn}

\begin{para}[\textbf{The underlying $|\Sigma|$-algebra $|A|$ of a $\Sigma$-algebra $A$}]
\label{ord_sig_para}
Let $\Sigma$ be a $\V$-enriched $\calS$-sorted signature with underlying classical $\calS$-sorted signature $|\Sigma|$ \eqref{ord_sig}. Every $\Sigma$-algebra $A$ has an \textbf{underlying $|\Sigma|$-algebra} $|A|$ with the same carrier object $A$ and, for each $\sigma \in \Sigma$ of type $(J, S, P)$ (where $\sigma^A : P \times A^J \to A_S$) and all $p \in |P|$, \[ \sigma_p^{|A|} := \sigma^A(p, -) : A^J \to A_S. \] If we instead write $\sigma^A : P \to \left[A^J, A_S\right]$, then $\sigma_p^{|A|} = \sigma^A(p) : A^J \to A_S$. Note that $\sigma_p^{|A|}$ is $\V$-admissible because $\V$ admits constant morphisms \eqref{given_data}.

In fact, we have a canonical fully faithful $\V$-functor $|-|^\Sigma : \Sigma\Alg \to |\Sigma|\Alg$: given $\Sigma$-algebras $A$ and $B$, the structural isomorphism 
\[ |-|^\Sigma_{AB} : \Sigma\Alg(A, B) \xrightarrow{\sim} |\Sigma|\Alg\left(|A|, |B|\right) \] is uniquely determined by the universal property of $|\Sigma|\Alg\left(|A|, |B|\right)$ as the pairwise equalizer of the family of parallel pairs \eqref{pp_eq} for $|\Sigma|$, $|A|$, and $|B|$. In detail, for each $\sigma \in \Sigma$ of type $(J, S, P)$ and each $p \in |P|$, the structural morphism $U^\Sigma_{A, B} : \Sigma\Alg(A, B) \rightarrowtail \V^\calS(A, B)$ equalizes the following parallel pair of \eqref{pp_eq} for $\sigma_p$, $|A|$, and $|B|$:
\begin{equation*}
\V^\calS(A, B) \xrightarrow{(-)^J_{AB}} \left[A^J, B^J\right] \xrightarrow{\left[1, \sigma_p^{|B|}\right]} \left[A^J, B_S\right]
\end{equation*}
\begin{equation*}
\V^\calS(A, B) \xrightarrow{\left(\pi_S\right)_{AB}} \left[A_S, B_S\right] \xrightarrow{\left[\sigma_p^{|A|}, 1\right]} \left[A^J, B_S\right],
\end{equation*} 
since it equalizes the parallel pair \eqref{pp_eq} for $\sigma$, $A$, and $B$. We thus obtain a unique $\V$-morphism $|-|^\Sigma_{AB} : \Sigma\Alg(A, B) \to |\Sigma|\Alg\left(|A|, |B|\right)$ satisfying $U^{|\Sigma|}_{|A||B|} \circ |-|^\Sigma_{AB} = U^\Sigma_{AB}$, which is a strong monomorphism because $U^\Sigma_{AB}$ is one. One readily verifies that the $\V$-morphism $|-|^\Sigma_{AB}$ is also surjective and thus epimorphic (because $|-| : \V \to \Set$ is faithful), so that $|-|^\Sigma_{AB}$ is an isomorphism. The structural isomorphisms $|-|^\Sigma_{AB}$ ($A, B \in \Sigma\Alg$) then readily assemble into a fully faithful $\V$-functor $|-|^\Sigma : \Sigma\Alg \to |\Sigma|\Alg$.  
\end{para}

\begin{para}[\textbf{$\Sigma$-terms in context and their interpretations}]
\label{syntax_para}
Let $\Sigma$ be a classical $\calS$-sorted signature \eqref{signature}, and let $\Var_S$ ($S \in \calS$) be pairwise disjoint countably infinite sets of variable symbols. An \emph{$\calS$-sorted variable} is an expression of the form $v : S$, for a sort $S \in \calS$ and a variable symbol $v \in \Var_S$. An \emph{$\calS$-sorted variable context}, or simply a \emph{context}, is a finite (possibly empty) unordered list of $\calS$-sorted variables, typically written $v_1 : S_1, \ldots, v_n : S_n$. Given a context $\vbar \equiv v_1 : S_1, \ldots, v_n : S_n$, we recursively define, for each sort $S \in \calS$, the set $\Term(\Sigma; \vbar)_S$ of \emph{$\Sigma$-terms of sort $S$ in context $\vbar$}, by the following clauses:
\begin{enumerate}[leftmargin=*]
\item For each $1 \leq i \leq n$, the expression $[\vbar \vdash v_i : S_i]$ is a $\Sigma$-term of sort $S_i$ in context $\vbar$.
\item Let $\sigma \in \Sigma$ be an operation symbol with input sorts $J \in \N_\calS$ and output sort $T$, let $T_1, \ldots, T_m$ be the finitely many sorts with $J\left(T_j\right) = n_j > 0$ for each $1 \leq j \leq m$, and for each $1 \leq j \leq m$ and each $1 \leq \ell \leq n_j$, let $[\vbar \vdash t_{j, \ell} : T_j]$ be a $\Sigma$-term of sort $T_j$ in context $\vbar$. Then the expression $[\vbar \vdash \sigma\left(t_{1, 1}, \ldots, t_{1, n_1}, \ldots, t_{m, 1}, \ldots, t_{m, n_m}\right) : T]$ is a $\Sigma$-term of sort $T$ in context $\vbar$.    
\end{enumerate} 
Given a context $\vbar$ and a sort $S \in \calS$, a \emph{syntactic $\Sigma$-equation of sort $S$ in context $\vbar$} is an ordered pair of $\Sigma$-terms of sort $S$ in context $\vbar$, and we write such an equation as $[\vbar \vdash s \doteq t : S]$. A \emph{syntactic $\Sigma$-equation in context $\vbar$} is a syntactic $\Sigma$-equation of sort $S$ in context $\vbar$ for some sort $S \in \calS$, while a \emph{syntactic $\Sigma$-equation (in context)} is a syntactic $\Sigma$-equation in context $\vbar$ for some context $\vbar$.  

Let $A$ be a $\Sigma$-algebra. Given a context $\vbar \equiv v_1 : S_1, \ldots, v_n : S_n$, we recursively define, for each sort $S \in \calS$, the \emph{interpretation in $A$} of each $\Sigma$-term $[\vbar \vdash t : S]$ of sort $S$ in context $\vbar$, which shall be a $\V$-morphism
\[ \left[\vbar \vdash t : S\right]^A : A_{S_1} \times \ldots \times A_{S_n} \to A_S, \] by the following clauses:
\begin{enumerate}[leftmargin=*]
\item For each $1 \leq i \leq n$, we define \[ \left[\vbar \vdash v_i : S_i\right]^A : A_{S_1} \times \ldots \times A_{S_n} \to A_{S_i} \] to be the product projection onto $A_{S_i}$.

\item Let $\sigma \in \Sigma$ be an operation symbol with input sorts $J \in \N_\calS$ and output sort $T$, let $T_1, \ldots, T_m$ be the finitely many sorts with $J\left(T_j\right) = n_j > 0$ for each $1 \leq j \leq m$, and for each $1 \leq j \leq m$ and each $1 \leq \ell \leq n_j$, let $[\vbar \vdash t_{j, \ell} : T_j]$ be a $\Sigma$-term of sort $T_j$ in context $\vbar$. Then we define 
\[ [\vbar \vdash \sigma\left(t_{1, 1}, \ldots, t_{1, n_1}, \ldots, t_{m, 1}, \ldots, t_{m, n_m}\right) : T]^A : A_{S_1} \times \ldots \times A_{S_n} \to A_T \] to be the following composite $\V$-morphism:
\[ A_{S_1} \times \ldots \times A_{S_n} \xrightarrow{\left\langle \left[\vbar \vdash t_{1, 1} : T_1\right]^A, \ldots, \left[\vbar \vdash t_{m, n_m} : T_m\right]^A\right\rangle} A^J \xrightarrow{\sigma^A} A_T. \]
\end{enumerate} 
We say that a $\Sigma$-algebra $A$ \emph{satisfies} a syntactic $\Sigma$-equation $[\vbar \vdash s \doteq t : S]$ of sort $S$ in context $\vbar$ if
\begin{equation}
\label{sat_synt}
\left[\vbar \vdash s : S\right]^A = \left[\vbar \vdash t : S\right]^A : A_{S_1} \times \ldots \times A_{S_n} \to A_S.
\end{equation} 
Given a morphism of $\Sigma$-algebras $f : A \to B$ and a sort $S \in \calS$, one readily verifies by induction on the $\Sigma$-term $[\vbar \vdash t : S]$ of sort $S$ in context $\vbar$ that the following diagram commutes:
\begin{equation}
\label{synt_to_nat_square}
\begin{tikzcd}
	{A_{S_1} \times \ldots \times A_{S_n}} &&& {B_{S_1} \times \ldots \times B_{S_n}} \\
	\\
	{A_S} &&& {B_S}.
	\arrow["{f_{S_1} \times \ldots \times f_{S_n}}", from=1-1, to=1-4]
	\arrow["{\left[\vbar \vdash t : S\right]^A}"', from=1-1, to=3-1]
	\arrow["{\left[\vbar \vdash t : S\right]^B}", from=1-4, to=3-4]
	\arrow["{f_S}"', from=3-1, to=3-4]
\end{tikzcd}
\end{equation}
For brevity, we shall often write $t^A$ instead of $\left[\vbar \vdash t : S\right]^A$. 
\end{para}

\noindent We now come to the central definition of this section.

\begin{defn}
\label{enriched_theory}
A \textbf{$\V$-enriched $\calS$-sorted equational theory} is a pair $\calT = (\Sigma, \calE)$ consisting of a $\V$-enriched $\calS$-sorted signature $\Sigma$ and a set $\calE$ of syntactic $|\Sigma|$-equations in context \eqref{syntax_para}, where $|\Sigma|$ is the underlying classical $\calS$-sorted signature of $\Sigma$ \eqref{ord_sig}. A $\Sigma$-algebra $A$ is a \textbf{$\calT$-algebra} if the underlying $|\Sigma|$-algebra $|A|$ \eqref{ord_sig_para} satisfies each syntactic $|\Sigma|$-equation in context in $\calE$ \eqref{sat_synt}. We write $\calT\Alg$ for the full sub-$\V$-category of the $\V$-category $\Sigma\Alg$ consisting of the $\calT$-algebras, so that $\calT\Alg$ is equipped with the faithful $\V$-functor $U^\calT : \calT\Alg \to \V^\calS$ obtained by restricting the faithful $\V$-functor $U^\Sigma : \Sigma\Alg \to \V^\calS$. We can thus regard $\calT\Alg$ as a $\V$-category over $\V^\calS$. When $\calS$ is a singleton, we instead say \textbf{$\V$-enriched single-sorted equational theory}.
\end{defn}

\noindent In \S\ref{monotop_subsection}, we compare the $\V$-enriched $\calS$-sorted equational theories of Definition \ref{enriched_theory} with those studied in \cite{Initial_algebras} (where $\V$ is assumed to be \emph{topological over $\Set$} but not necessarily cartesian closed). In \S\ref{examples_section} we shall also indicate many examples of $\V$-enriched $\calS$-sorted equational theories in the literature.

\section{Strongly finitary $\V$-monads on $\V^\calS$}
\label{sf_section}

In this section, we introduce strongly finitary $\V$-monads on $\V^\calS$. 

\begin{para}[\textbf{The subcategory of arities $\N_\calS \hookrightarrow \V^\calS$}]
\label{N_S_para}
For each natural number $n \in \N$, we write $n \cdot 1$ for the $n^{\mathsf{th}}$ copower of the terminal object $1$ in $\V$. Since $|-| \cong \V(1, -) : \V \to \Set$, it follows that $\V$-morphisms $n \cdot 1 \to X$ ($n \in \N, X \in \ob\V$) are in bijective correspondence with $n$-tuples of elements of the underlying set $|X|$. 

We write $\N_\V \hookrightarrow \V$ for the (small) full sub-$\V$-category consisting of the objects $n \cdot 1$ ($n \in \N$). For each $n \in \N$ and each $X \in \ob\V$, note that $[n \cdot 1, X] \cong [1, X]^n \cong X^n$, the $n^{\mathsf{th}}$ power of $X$. We also write $\N_\V^\calS$ for the (small) full sub-$\V$-category $\prod_\calS \N_\V \hookrightarrow \prod_\calS \V = \V^\calS$. Finally, we write $\N_\calS$ for the (small) full sub-$\V$-category of $\N_\V^\calS$ (and hence of $\V^\calS$) consisting of the objects $\left(n_S \cdot 1\right)_{S \in \calS}$ of $\N_\V^\calS$ such that $n_S = 0$ for all but finitely many $S \in \calS$. It is in this way that we can regard the set $\N_\calS$ of functions $\calS \to \N$ with finite support as a small full sub-$\V$-category of $\V^\calS$. We shall often denote an object of $\N_\calS$ by $J = \left(n_{S_1}, \ldots, n_{S_m}\right)$, where $S_1, \ldots, S_m$ are the finitely many sorts with $n_{S_i} > 0$ ($1 \leq i \leq m$). Note that we have
\begin{equation}
\label{repr_eq} 
\V^\calS(J, X) = \prod_{S \in \calS}\left[n_S \cdot 1, X_S\right] \cong \prod_{S \in \calS} X_S^{n_S} = X^J
\end{equation}
$\V$-naturally in $X \in \V^\calS$, so that the object $J$ of $\N_\calS$ represents the $\V$-functor $(-)^J : \V^\calS \to \V$ of \ref{prod_not}. Thus, we have $\left|X^J\right| \cong \left|\V^\calS(J, X)\right|$, so we can and shall assume that the underlying set $\left|X^J\right|$ is the set of morphisms from $J$ to $X$ in $\V^\calS$. 

If $J = \left(n_{S_1}, \ldots, n_{S_m}\right) = \left(n_1, \ldots, n_m\right)$ and $X$ is an object of $\V^\calS$, then note that a morphism $f : J \to X$ of $\V^\calS$ is completely determined by, for each $1 \leq j \leq m$, an $n_{j}$-tuple of elements of the underlying set $\left|X_{S_j}\right|$, which we write as $\left(f_{S_j}(1), \ldots, f_{S_j}\left(n_j\right)\right)$. We shall then sometimes denote the morphism $f : J \to X$ of $\V^\calS$ by
\begin{equation}
\label{J_eq}
f = \left(f_{S_1}(1), \ldots, f_{S_1}\left(n_1\right), \ldots, f_{S_m}(1), \ldots, f_{S_m}\left(n_m\right)\right) : J \to X,
\end{equation} 
or more succinctly by $f = \left(f_{S_1}(1), \ldots, f_{S_m}\left(n_m\right)\right) : J \to X$. 

By \cite[6.2]{Initial_algebras}, the small full sub-$\V$-category $\N_\calS \hookrightarrow \V^\calS$ is \emph{dense} (in the enriched sense; see \cite[Chapter 5]{Kelly}), and thus it is a small \emph{subcategory of arities} in the sense of \cite[Definition 3.1]{Pres} and \cite[\S 3]{Struct}.
\end{para}

\noindent The following terminology was introduced by Kelly and Lack \cite[\S 3]{KellyLackstronglyfinitary} in the case where $\calS$ is a singleton (and $\V(1, -) : \V \to \Set$ is not assumed to be faithful). 

\begin{defn}
\label{sf_def}
A $\V$-endofunctor $T : \V^\calS \to \V^\calS$ is \textbf{strongly finitary} if it is the ($\V$-enriched) left Kan extension of its restriction along $\N_\calS \hookrightarrow \V^\calS$. 
A $\V$-monad $\T = (T, \eta, \mu)$ on $\V^\calS$ is \textbf{strongly finitary} if the $\V$-endofunctor $T : \V^\calS \to \V^\calS$ is strongly finitary. We write $\Mnd\left(\V^\calS\right)$ for the category of all $\V$-monads on $\V^\calS$ and $\Mnd_\sff\left(\V^\calS\right) \hookrightarrow \Mnd\left(\V^\calS\right)$ for the full subcategory consisting of the strongly finitary $\V$-monads. Given a $\V$-monad $\T$ on $\V^\calS$, we write $\T\Alg$ for the $\V$-category of $\T$-algebras and $U^\T : \T\Alg \to \V^\calS$ for the forgetful $\V$-functor (defined in \cite[II.1]{Dubucbook}), so that $\T\Alg$ may be regarded as a $\V$-category over $\V^\calS$. 
\end{defn}

\begin{para}[\textbf{The subcategory of arities $\N_\calS \hookrightarrow \V^\calS$ is eleutheric}]
\label{eleuth_para}
Writing $j : \N_\calS \hookrightarrow \V^\calS$ for the inclusion, we write $\Phi_{\N_\calS}$ for the class of small weights $\V^\calS(j-, X) : \N_\calS^\op \to \V$ ($X \in \ob\V^\calS$).
Since $\V$ is cartesian closed, we deduce from \cite[Theorem 6.8]{Initial_algebras} that the subcategory of arities $\N_\calS \hookrightarrow \V^\calS$ is \emph{eleutheric} \cite[Definition 3.3]{Pres}, meaning that each representable $\V$-functor $\V^\calS(J, -) : \V^\calS \to \V$ ($J \in \N_\calS$) preserves weighted $\Phi_{\N_\calS}$-colimits, or equivalently preserves ($\V$-enriched) left Kan extensions along $j : \N_\calS \hookrightarrow \V^\calS$.\footnote{\cite[Definition 3.3]{Pres} also requires that $\V^\calS$ have weighted $\Phi_{\N_\calS}$-colimits, but this is true in the present context because $\N_\calS$ is small and $\V^\calS$ is cocomplete (as a $\V$-category), since $\V$ is cocomplete.} 

We then deduce from \cite[Proposition 4.2]{Pres} that a $\V$-endofunctor $T : \V^\calS \to \V^\calS$ is strongly finitary in the sense of Definition \ref{sf_def} iff it is \emph{$\N_\calS$-ary} in the sense of \cite[Definition 4.1]{Pres}, meaning that it preserves weighted $\Phi_{\N_\calS}$-colimits, or equivalently preserves left Kan extensions along $j : \N_\calS \hookrightarrow \V^\calS$. As a consequence, the strongly finitary $\V$-endofunctors of $\V^\calS$ are closed under composition, and a $\V$-monad on $\V^\calS$ is strongly finitary iff it is \emph{$\N_\calS$-ary}, meaning that its underlying $\V$-endofunctor is $\N_\calS$-ary, i.e.~preserves left Kan extensions along $j : \N_\calS \hookrightarrow \V^\calS$. 
\end{para}

\noindent We conclude this section with the following central definition:

\begin{defn}
\label{monad_pres_def}
A $\V$-enriched $\calS$-sorted equational theory $\calT$ \textbf{presents} a $\V$-monad $\T$ on $\V^\calS$ if $\calT\Alg \cong \T\Alg$ in $\V\CAT/\V^\calS$. 
\end{defn}

\section{$\V$-enriched $\calS$-sorted equational theories present strongly finitary $\V$-monads on $\V^\calS$}
\label{diag_section}

Our first objective is to establish, using results of \cite{EP}, that every $\V$-enriched $\calS$-sorted equational theory presents a strongly finitary $\V$-monad on $\V^\calS$. To prove this, we shall first recall some relevant notions from \cite{EP}, specialized to the small subcategory of arities $\N_\calS \hookrightarrow \V^\calS$ \eqref{N_S_para}.
Recall that the $\V$-category $\V^\calS$ is \emph{tensored} and \emph{cotensored} \cite[\S 3.7]{Kelly}, with tensors and cotensors in $\V^\calS$ being formed \emph{pointwise} \cite[\S 3.3]{Kelly}: given an object $V$ of $\V$ and an object $X$ of $\V^\calS$, the tensor $V \tensor X$ in $\V^\calS$ is given by
\begin{equation}
\label{tensor_eq}
V \tensor X = \left(V \times X_S\right)_{S \in \calS},
\end{equation}
while the cotensor $[V, X]$ in $\V^\calS$ is given by 
\begin{equation}
\label{cotensor_eq}
[V, X] = \left(\left[V, X_S\right]\right)_{S \in \calS}.
\end{equation} 

\begin{para}[\textbf{Free-form $\N_\calS$-signatures and their algebras}]
\label{EP_sig_para}
A \emph{free-form $\N_\calS$-signature} \cite[Definition 5.1]{EP} is a set $\scrS$ of operation symbols equipped with an assignment to each operation symbol $\sigma \in \scrS$ of an \emph{arity} $J = J_\sigma \in \N_\calS$ and a \emph{parameter (object)} $P = P_\sigma \in \ob\V^\calS$.\footnote{We emphasize that here, the parameter object is an object of $\V^\calS$ rather than $\V$ (cf.~Definition \ref{signature}).} An \emph{$\scrS$-algebra} $A$ \cite[Definition 5.2]{EP} is an object $A$ of $\V^\calS$ equipped with, for each $\sigma \in \scrS$ of arity $J$ and parameter $P$, a $\V^\calS$-morphism
\[ \sigma^A : \V^\calS\left(J, A\right) \tensor P \to A, \] which, by \eqref{repr_eq}, may be equivalently written as
\[ \sigma^A : A^J \tensor P \to A. \]
Then by \eqref{tensor_eq}, the $\V^\calS$-morphism $\sigma^A : A^J \tensor P \to A$ is completely determined by, for each sort $S \in \calS$, a $\V$-morphism
\begin{equation}
\label{EP_sig_eq}
\sigma^A_S : A^J \times P_S \to A_S.
\end{equation} 
Given $\scrS$-algebras $A$ and $B$, a \emph{morphism of $\scrS$-algebras $f : A \to B$} \cite[Definition 5.2]{EP} is a morphism $f : A \to B$ of $\V^\calS$ that makes the following diagram commute for each $\sigma \in \scrS$ of arity $J$ and parameter $P$:
\[\begin{tikzcd}
	{A^J \tensor P} &&& {B^J \tensor P} \\
	\\
	{A} &&& {B}.
	\arrow["{f^J \tensor 1_P}", from=1-1, to=1-4]
	\arrow["{\sigma^A}"', from=1-1, to=3-1]
	\arrow["{\sigma^B}", from=1-4, to=3-4]
	\arrow["{f}"', from=3-1, to=3-4]
\end{tikzcd}\] 
We then have an ordinary category $\scrS\Alg_0$ of $\scrS$-algebras that (as in \cite[Definition 5.2]{EP}) underlies a $\V$-category $\scrS\Alg$, which is equipped with a faithful $\V$-functor $U^\scrS : \scrS\Alg \to \V^\calS$ that sends each $\scrS$-algebra $A$ to the object $A \in \ob\V^\calS$. We may thus regard $\scrS\Alg$ as a $\V$-category over $\V^\calS$. 
\end{para}

\begin{para}[\textbf{Diagrammatic $\N_\calS$-presentations and their algebras}]
\label{EP_eq_para}
Given a free-form $\N_\calS$-signature $\scrS$, a \emph{natural $\scrS$-operation}\footnote{In \cite[Definition 5.10]{EP} the concept of \emph{$\V$-natural $\scrS$-operation} is defined, but $\V$-naturality reduces to ordinary naturality in the present context because $\V(1, -) : \V \to \Set$ is faithful.} \cite[Definition 5.10]{EP} is a natural transformation $\omega : \V^\calS\left(J, U^\scrS-\right) \tensor P \Longrightarrow U^\scrS : \scrS\Alg \to \V^\calS$ with specified \emph{arity} $J \in \N_\calS$ and \emph{parameter (object)} $P \in \ob\V^\calS$. Explicitly, a natural $\scrS$-operation $\omega$ consists of morphisms $\omega_A : A^J \tensor P \to A$ ($A \in \ob\scrS\Alg$) of $\V^\calS$ such that the following diagram commutes for each morphism $f : A \to B$ of $\scrS\Alg$:
\[\begin{tikzcd}
	{A^J \tensor P} &&& {B^J \tensor P} \\
	\\
	{A} &&& {B}.
	\arrow["{f^J \tensor 1_P}", from=1-1, to=1-4]
	\arrow["{\omega_A}"', from=1-1, to=3-1]
	\arrow["{\omega_B}", from=1-4, to=3-4]
	\arrow["{f}"', from=3-1, to=3-4]
\end{tikzcd}\]  
A \emph{diagrammatic $\scrS$-equation} \cite[Definition 5.10]{EP}, written $\omega \doteq \nu$, is a pair of natural $\scrS$-operations $\omega, \nu : \V^\calS\left(J, U^\scrS-\right) \tensor P \Longrightarrow U^\scrS$ with the same arity $J$ and parameter $P$. An $\scrS$-algebra $A$ \emph{satisfies} a diagrammatic $\scrS$-equation $\omega \doteq \nu$ if $\omega_A = \nu_A : A^J \tensor P \to A$. Finally, a \emph{diagrammatic $\N_\calS$-presentation} \cite[Definition 5.12]{EP} is a pair $\scrT = (\scrS, \scrE)$ consisting of a free-form $\N_\calS$-signature $\scrS$ and a set $\scrE$ of diagrammatic $\scrS$-equations. We write $\scrT\Alg$ for the full sub-$\V$-category of $\scrS\Alg$ consisting of the \emph{$\scrT$-algebras}, i.e.~the $\scrS$-algebras that satisfy all of the diagrammatic $\scrS$-equations in $\scrE$. The forgetful $\V$-functor $U^\scrS : \scrS\Alg \to \V^\calS$ restricts to a forgetful $\V$-functor $U^\scrT : \scrT\Alg \to \V^\calS$, and we may thus regard $\scrT\Alg$ as a $\V$-category over $\V^\calS$.     
\end{para} 

\noindent The following result establishes that each $\V$-enriched $\calS$-sorted equational theory has the same ``expressive power'' as some diagrammatic $\N_\calS$-presentation.

\begin{prop}
\label{equiv_presentations_prop1}
\
\begin{enumerate}[leftmargin=*]
\item For each $\V$-enriched $\calS$-sorted signature $\Sigma$ \eqref{signature}, there is a free-form $\N_\calS$-signature $\scrS_\Sigma$ \eqref{EP_sig_para} such that $\Sigma\Alg \cong \scrS_\Sigma\Alg$ in $\VCAT/\V^\calS$. 

\item For each $\V$-enriched $\calS$-sorted signature $\Sigma$ \eqref{signature} and each syntactic $|\Sigma|$-equation $[\vbar \vdash s \doteq t : S]$ \eqref{syntax_para}, there is a diagrammatic $\scrS_\Sigma$-equation $\omega_s \doteq \omega_t$ with the property that for every $\Sigma$-algebra $A$, the underlying $|\Sigma|$-algebra $|A|$ \eqref{ord_sig_para} satisfies $[\vbar \vdash s \doteq t : S]$ \eqref{sat_synt} iff the $\scrS_\Sigma$-algebra corresponding to $A$ by (1) satisfies $\omega_s \doteq \omega_t$. 
 
\item For each $\V$-enriched $\calS$-sorted equational theory $\calT$, there is a diagrammatic $\N_\calS$-presentation $\scrT_\calT$ such that $\calT\Alg \cong \scrT_\calT\Alg$ in $\VCAT/\V^\calS$.
\end{enumerate} 
\end{prop}

\begin{proof}
(3) readily follows from (1) and (2), and we obtain (1) from \cite[Proposition 6.5(1)]{Initial_algebras}. We obtain (2) by combining \cite[Proposition 4.1.15(3)]{Initial_algebras} with \cite[Proposition 6.5(2)]{Initial_algebras}. 
\end{proof}

\noindent To prove the main result of this section (Theorem \ref{th_to_mnd}), we also require the following lemma.

\begin{lem}
\label{bounded_lem}
For each $J \in \N_\calS$, the representable $\V$-functor $\V^\calS(J, -) \cong (-)^J : \V^\calS \to \V$ \eqref{prod_not} preserves (conical) filtered colimits.
\end{lem}

\begin{proof}
With $J = \left(n_{S_1}, \ldots, n_{S_m}\right)$ and $n_i := n_{S_i}$ ($1 \leq i \leq m$), by \eqref{prod_not_eq} and \eqref{repr_eq} and we have \[ \V^\calS(J, X) \cong X^J \cong X_{S_1}^{n_1} \times \ldots \times X_{S_m}^{n_m}, \] $\V$-naturally in $X \in \V^\calS$. So $\V^\calS(J, -) : \V^\calS \to \V$ is isomorphic to the composite $\V$-functor
\[ \V^\calS \xrightarrow{\left\langle \pi_{S_1}, \ldots, \pi_{S_m}\right\rangle} \V^m \xrightarrow{(-)^{n_1} \times \ldots \times (-)^{n_m}} \V^m \xrightarrow{\times} \V, \] where $\times : \V^m \to \V$ is the product $\V$-functor given by $\left(X_1, \ldots, X_m\right) \mapsto X_1 \times \ldots \times X_m$ ($X_1, \ldots, X_m \in \V$). To prove the result, it therefore suffices to show that each functor $(-)^n : \V \to \V$ ($n \geq 1$) preserves filtered colimits, and that $\times : \V^m \to \V$ preserves filtered colimits. Since $\V$ is cartesian closed, both claims follow from \cite[3.8]{Kellystr} (and its proof).
\end{proof}

\begin{para}[\textbf{The subcategory of arities $\N_\calS \hookrightarrow \V^\calS$ is bounded and eleutheric}]
\label{bounded_para}
Since $\V$ is cartesian closed, we know from \ref{eleuth_para} that the subcategory of arities $\N_\calS \hookrightarrow \V^\calS$ is \emph{eleutheric} \cite[Definition 3.3]{Pres}. From Lemma \ref{bounded_lem} we also know that the subcategory of arities $\N_\calS \hookrightarrow \V^\calS$ is \emph{bounded} in the sense of \cite[Definitions 6.1.6 and 6.1.10]{Pres}. Then, in view of \ref{eleuth_para}, \cite[Theorem 5.20]{EP} specializes to the present context to yield the following: 
\end{para}

\begin{theo}[\cite{EP}]
\label{EP_thm}
For every diagrammatic $\N_\calS$-presentation $\scrT$ \eqref{EP_eq_para}, there is a strongly finitary $\V$-monad $\T_\scrT$ on $\V^\calS$ such that $\scrT\Alg \cong \T_\scrT\Alg$ in $\V\CAT/\V^\calS$. \qed
\end{theo}

\noindent We may now prove the central result of this section:

\begin{theo}
\label{th_to_mnd}
Under the assumptions of \ref{given_data}, every $\V$-enriched $\calS$-sorted equational theory presents a strongly finitary $\V$-monad on $\V^\calS$.
\end{theo} 

\begin{proof}
Given a $\V$-enriched $\calS$-sorted equational theory $\calT$, by Proposition \ref{equiv_presentations_prop1} there is a diagrammatic $\N_\calS$-presentation $\scrT$ such that $\calT\Alg \cong \scrT\Alg$ in $\VCAT/\V^\calS$. Then by Theorem \ref{EP_thm} there is a strongly finitary $\V$-monad $\T$ on $\V^\calS$ such that $\T\Alg \cong \scrT\Alg \cong \calT\Alg$ in $\VCAT/\V^\calS$, as desired. 
\end{proof}

\section{Strongly finitary $\V$-monads on $\V^\calS$ are presented by $\V$-enriched $\calS$-sorted equational theories}
\label{converse_section}

In this section, we shall establish in Theorem \ref{mnd_has_pres_thm} the converse of Theorem \ref{th_to_mnd}, namely that every strongly finitary $\V$-monad on $\V^\calS$ is presented by a $\V$-enriched $\calS$-sorted equational theory. Recalling that we write $j : \N_\calS \hookrightarrow \V^\calS$ for the inclusion, the idea for the proof of Theorem \ref{mnd_has_pres_thm} is as follows: first, every strongly finitary $\V$-monad $\T = (T, \eta, \mu)$ on $\V^\calS$ determines an associated \emph{$\N_\calS$-relative $\V$-monad} $Tj$ on $\V^\calS$ \eqref{rel_mnd_para}. There is a $\V$-category $Tj\Alg$ over $\V^\calS$ of \emph{algebras} \eqref{rel_mnd_para} for the $\N_\calS$-relative $\V$-monad $Tj$, and it turns out that $Tj\Alg$ is isomorphic (in $\V\CAT/\V^\calS$) to the $\V$-category of algebras for a certain $\V$-enriched $\calS$-sorted equational theory $\calT_\T$ determined by $\T$ \eqref{mnd_th_def}. Since the subcategory of arities $\N_\calS \hookrightarrow \V^\calS$ is bounded and eleutheric \eqref{bounded_para}, we are then able to deduce from \cite{EP} that $\T\Alg \cong Tj\Alg$ in $\V\CAT/\V^\calS$, which yields Theorem \ref{mnd_has_pres_thm}.  

Combining Theorems \ref{th_to_mnd} and \ref{mnd_has_pres_thm} will then allow us to establish Theorem \ref{variety_thm}, which states that the category $\Mnd_\sff\left(\V^\calS\right)$ of strongly finitary $\V$-monads on $\V^\calS$ is dually equivalent to the category $\Var\left(\V^\calS\right)$ of \emph{(finitary) $\V$-enriched $\calS$-sorted varieties} \eqref{variety}.

\begin{para}[\textbf{$\N_\calS$-relative $\V$-monads and their algebras}]
\label{rel_mnd_para}
Let $\T = (T, \eta, \mu)$ be a $\V$-monad on $\V^\calS$. The composite $\V$-functor $Tj : \N_\calS \to \V^\calS$ carries the structure of an \emph{$\N_\calS$-relative $\V$-monad\footnote{Relative monads were introduced (in the unenriched setting) in \cite{Altenkirchmonads}.} (on $\V^\calS$)} \cite[Definition 8.1]{Pres} that we also denote by $Tj$. In detail, what this means is that we have the $\N_\calS$-indexed family $\left(\eta_J : J \to TJ\right)_{J \in \N_\calS}$ of morphisms of $\V^\calS$, together with the $\N_\calS^2$-indexed family \[ \left(m_{JK} : \V^\calS(J, TK) \to \V^\calS(TJ, TK)\right)_{J, K \in \N_\calS} \] of morphisms of $\V^\calS$, where
\[ m_{JK} = \left(\V^\calS(J, TK) \xrightarrow{T_{J, TK}} \V^\calS(TJ, TTK) \xrightarrow{\V^\calS\left(1, \mu_K\right)} \V^\calS(TJ, TK)\right) \] for all $J, K \in \N_\calS$, and these data make the three diagrams of \cite[Definition 8.1]{Pres} commute. For each morphism $f : J \to TK$ ($J, K \in \N_\calS$) of $\V^\calS$, we write 
\begin{equation}
\label{*_eq}
f^* := \left(TJ \xrightarrow{Tf} TTK \xrightarrow{\mu_K} TK\right),
\end{equation}
so that the underlying function of $m_{JK}$ is given by $f \mapsto f^*$ ($f \in \V^\calS(J, TK)$).

As in \cite[Definition 8.4]{Pres}, an \emph{algebra} for the $\N_\calS$-relative $\V$-monad $Tj$, or a \emph{$Tj$-algebra}, is an object $A$ of $\V^\calS$ equipped with, for each $J \in \N_\calS$, a $\V$-morphism
\begin{equation}
\label{rel_mor}
\sigma_J^A : \V^\calS(J, A) \to \V^\calS(TJ, A),
\end{equation} 
such that the following diagrams commute for all $J, K \in \N_\calS$: 
\begin{equation}
\label{rel_1}
\begin{tikzcd}
	{\V^\calS(J, A)} && {\V^\calS(TJ, A)} && \V^\calS(J, A)
	\arrow["{\sigma_J^A}", from=1-1, to=1-3]
	\arrow["{\V^\calS\left(\eta_J, 1\right)}", from=1-3, to=1-5]
	\arrow["{1}"', curve={height=18pt}, from=1-1, to=1-5]
\end{tikzcd}
\end{equation}
\begin{equation}
\label{rel_2}
\begin{tikzcd}
	{\V^\calS(J, TK) \times \V^\calS(K, A)} &&& {\V^\calS(TJ, TK) \times \V^\calS(TK, A)} \\
	{V^\calS(J, TK) \times \V^\calS(TK, A)} && {\V^\calS(J, A)} & {\V^\calS(TJ, A)},
	\arrow["{m_{JK} \times \sigma_K^A}", from=1-1, to=1-4]
	\arrow["c", from=1-4, to=2-4]
	\arrow["{1 \times \sigma_K^A}"', from=1-1, to=2-1]
	\arrow["c"', from=2-1, to=2-3]
	\arrow["{\sigma_J^A}"', from=2-3, to=2-4]
\end{tikzcd}
\end{equation}
where each $c$ is a composition morphism of the $\V$-category $\V^\calS$. As in \cite[Definition 8.8]{EP}, there is a $\V$-category $Tj\Alg$ over $\V^\calS$ of $Tj$-algebras. 

For each sort $S \in \calS$, we have the $\V$-morphism
\begin{equation}
\label{sigma_comp_eq}
\sigma_{J, S}^A : (TJ)_S \times A^J \to A_S
\end{equation}
given as the composite
\[ (TJ)_S \times A^J \xrightarrow{\sim} \V^\calS(J, A) \times (TJ)_S \xrightarrow{\sigma_J^A \times 1} \V^\calS(TJ, A) \times (TJ)_S \xrightarrow{\pi_S \times 1} \left[(TJ)_S, A_S\right] \times (TJ)_S \xrightarrow{\ev} A_S, \] where $\pi_S : \V^\calS(TJ, A) \to \left[(TJ)_S, A_S\right]$ is a projection morphism and $\ev : \left[(TJ)_S, A_S\right] \times (TJ)_S \to A_S$ is an evaluation morphism \eqref{given_data}. Since $|-| : \V \to \Set$ is faithful, the commutativity of \eqref{rel_1} for each $J = \left(n_{S_1}, \ldots, n_{S_m}\right) \in \N_\calS$ is readily seen to be equivalent to the satisfaction of the following equality for all $1 \leq i \leq m$, all $1 \leq j \leq n_{S_i}$, and all $x : J \to A$:
\begin{equation}
\label{rel_1'}
\left(\sigma_{J, S_i}^A\right)\left(\eta_{J, S_i}(j), x\right) = x_{S_i}(j) \in \left|A_{S_i}\right|,
\end{equation}
using the notation of \eqref{J_eq} and recalling that $\eta_{J, S_i} : \left(n_{S_i} \cdot 1\right) \to (TJ)_{S_i}$. 

Similarly, the commutativity of \eqref{rel_2} for all $J = \left(n_{S_1}, \ldots, n_{S_m}\right) = \left(n_1, \ldots, n_m\right), K \in \N_\calS$ is equivalent to the satisfaction of the following equality for all $k : J \to TK$, all $S \in \calS$, all $p \in \left|(TJ)_S\right|$, and all $x : K \to A$:
\begin{equation}
\label{rel_2'}
\left(\sigma_{K, S}^A\right)\left(k^*_S(p), x\right) = \left(\sigma_{J, S}^A\right)\left(p, \left(\left(\sigma_{K, S_1}^A\right)\left(k_{S_1}(1), x\right), \ldots, \left(\sigma_{K, S_{m}}^A\right)\left(k_{S_m}(n_m), x\right)\right)\right) \in \left|A_S\right|,
\end{equation}
using the notation of \eqref{J_eq} and recalling that $k^* : TJ \to TK$ \eqref{*_eq}.

\end{para}

\begin{para}[\textbf{The variable context $\vec{v} : J$}]
\label{notation_para}
For each $J = \left(n_{S_1}, \ldots, n_{S_m}\right) \in \N_\calS$, we write $\vbar : J$ for the following $\calS$-sorted variable context: 
\[ v_{1, 1} : S_1, \ldots, v_{1, n_{S_1}} : S_1, \ldots, v_{m, 1} : S_m, \ldots, v_{m, n_{S_m}} : S_m \]
consisting of $n_{S_i}$ pairwise distinct $\calS$-sorted variables of sort $S_i$ for each $1 \leq i \leq m$.  
\end{para}

\noindent The considerations of \ref{rel_mnd_para} now motivate the following definition:

\begin{defn}[\textbf{The $\V$-enriched $\calS$-sorted equational theory associated to a $\V$-monad on $\V^\calS$}]
\label{mnd_th_def}
Given a $\V$-monad $\T = (T, \eta, \mu)$ on $\V^\calS$, we define an associated $\V$-enriched $\calS$-sorted equational theory $\calT_\T = \left(\Sigma_\T, \calE_\T\right)$ as follows. 

For each $J \in \N_\calS$ and each $S \in \calS$, we consider an operation symbol $\sigma_{J, S}$ of input sorts $J$, output sort $S$, and parameter object $(TJ)_S \in \ob\V$, and we define the $\V$-enriched $\calS$-sorted signature $\Sigma_\T$ to consist of precisely these operation symbols $\sigma_{J, S}$ ($J \in \N_\calS, S \in \calS$).

We define the set $\calE_\T$ of syntactic $\left|\Sigma_\T\right|$-equations in context as follows. 
\begin{itemize}[leftmargin=*]
\item For each morphism $k : J \to TK$ of $\V^\calS$ with $J, K \in \N_\calS$ and $J = \left(n_{S_1}, \ldots, n_{S_m}\right) = \left(n_1, \ldots, n_m\right)$, each sort $S \in \calS$, and each element $p \in \left|(TJ)_S\right|$, we consider the following syntactic $\left|\Sigma_\T\right|$-equation of sort $S$ in context $\vbar : K$ \eqref{notation_para}:
\begin{equation}
\label{k_eq}
\left[\vbar : K \vdash \left(\sigma_{K, S}\right)_{k^*_S(p)}\left(\vbar\right) \doteq \left(\sigma_{J, S}\right)_p\left(\left(\sigma_{K, S_1}\right)_{k_{S_1}(1)}\left(\vbar\right), \ldots, \left(\sigma_{K, S_{m}}\right)_{k_{S_m}(n_m)}\left(\vbar\right)\right) : S\right],
\end{equation}
using the notation of \eqref{J_eq} and recalling that $k^* : TJ \to TK$ \eqref{*_eq}.

\item For each $J = \left(n_{S_1}, \ldots, n_{S_m}\right) \in \N_\calS$, each $1 \leq i \leq m$, and each $1 \leq j \leq n_{S_i}$, we also consider the following syntactic $\left|\Sigma_\T\right|$-equation of sort $S_i$ in context $\vbar : J$ \eqref{notation_para}:
\begin{equation}
\label{eta_eq}
\left[\vbar : J \vdash \left(\sigma_{J, S_i}\right)_{\eta_{J, S_i}(j)}\left(\vbar\right) \doteq v_{i, j} : S_i\right],
\end{equation} 
using the notation of \eqref{J_eq} and recalling that $\eta_{J, S_i} : \left(n_{S_i} \cdot 1\right) \to (TJ)_{S_i}$.
\end{itemize}
We then declare that $\calE_\T$ consists of precisely the syntactic $\left|\Sigma_\T\right|$-equations of the form \eqref{k_eq} and \eqref{eta_eq}. This completes the definition of the $\V$-enriched $\calS$-sorted equational theory $\calT_\T = \left(\Sigma_\T, \calE_\T\right)$.
\end{defn}

\noindent We may now prove the following converse of Theorem \ref{th_to_mnd}:

\begin{theo}
\label{mnd_has_pres_thm}
Under the assumptions of \ref{given_data}, every strongly finitary $\V$-monad $\T$ on $\V^\calS$ is presented by the $\V$-enriched $\calS$-sorted equational theory $\calT_\T = \left(\Sigma_\T, \calE_\T\right)$ \eqref{mnd_th_def}.
\end{theo}

\begin{proof}
Since $\T$ is strongly finitary (i.e.~$\N_\calS$-ary, \ref{eleuth_para}) and the subcategory of arities $\N_\calS \hookrightarrow \V^\calS$ is bounded and eleutheric \eqref{bounded_para}, we deduce from \cite[Remark 8.9]{EP} that $\T\Alg$ is isomorphic in $\V\CAT/\V^\calS$ to the $\V$-category $Tj\Alg$ of algebras for the corresponding $\N_\calS$-relative $\V$-monad $Tj$ \eqref{rel_mnd_para}. Moreover, it readily follows from \cite[Proposition 8.7]{EP} and the discussion in \ref{rel_mnd_para} that the $\V$-category $Tj\Alg$ is isomorphic (in $\V\CAT/\V^\calS$) to the full sub-$\V$-category of $\Sigma_\T\Alg$ consisting of the $\Sigma_\T$-algebras that satisfy all of the equalities \eqref{rel_1'} and \eqref{rel_2'}. But it is immediate that these $\Sigma_\T$-algebras are precisely the $\calT_\T$-algebras, so that $\T\Alg \cong \calT_\T\Alg$ in $\V\CAT/\V^\calS$, as desired.
\end{proof}

\noindent For the record, we may now deduce the following converse to Proposition \ref{equiv_presentations_prop1}(3).

\begin{cor}
\label{diag_cor}
Under the assumptions of \ref{given_data}, for every diagrammatic $\N_\calS$-presentation $\scrT$, there is a $\V$-enriched $\calS$-sorted equational theory $\calT$ such that $\scrT\Alg \cong \calT\Alg$ in $\V\CAT/\V^\calS$.
\end{cor}

\begin{proof}
By Theorem \ref{EP_thm}, there is a strongly finitary $\V$-monad $\T$ on $\V^\calS$ such that $\scrT\Alg \cong \T\Alg$ in $\V\CAT/\V^\calS$. Then by Theorem \ref{mnd_has_pres_thm}, there is a $\V$-enriched $\calS$-sorted equational theory $\calT_\T$ such that $\scrT\Alg \cong \T\Alg \cong \calT_\T\Alg$ in $\V\CAT/\V^\calS$, as desired.
\end{proof}

\begin{defn}
\label{variety}
A $\V$-category over $\V^\calS$ is a \textbf{(finitary) $\V$-enriched $\calS$-sorted variety} if it is of the form $\calT\Alg$ for some $\V$-enriched $\calS$-sorted equational theory $\calT$. We define the (ordinary) category $\Var\left(\V^\calS\right)$ of $\V$-enriched $\calS$-sorted varieties to be the full subcategory of $\V\CAT/\V^\calS$ consisting of the $\V$-enriched $\calS$-sorted varieties. 
\end{defn}

\begin{defn}
\label{sf_monadic}
Given a $\V$-functor $U : \A \to \V^\calS$, so that $\A$ may be regarded as a $\V$-category over $\V^\calS$, we say that $U$ (or $\A$) is \textbf{strictly $\N_\calS$-monadic} \cite[Definition 4.10]{Pres}, or \textbf{strongly finitary strictly monadic}, if there is a strongly finitary $\V$-monad $\T$ on $\V^\calS$ such that $\A \cong \T\Alg$ in $\V\CAT/\V^\calS$. We write $\SFF\Monadic^!$ for the full subcategory of $\V\CAT/\V^\calS$ consisting of the strongly finitary strictly monadic $\V$-categories over $\V^\calS$.
\end{defn}

\begin{theo}
\label{variety_thm}
Under the assumptions of \ref{given_data}, the category $\Var\left(\V^\calS\right)$ of $\V$-enriched $\calS$-sorted varieties is dually equivalent to the category $\Mnd_\sff\left(\V^\calS\right)$ of strongly finitary $\V$-monads on $\V^\calS$:
\[ \Var\left(\V^\calS\right) \simeq \Mnd_\sff\left(\V^\calS\right)^\op. \]  
\end{theo}

\begin{proof}
By Theorems \ref{th_to_mnd} and \ref{mnd_has_pres_thm}, the repletion of $\Var\left(\V^\calS\right)$ in $\V\CAT/\V^\calS$ is the full subcategory $\SFF\Monadic^!$ of $\V\CAT/\V^\calS$ \eqref{sf_monadic}. So $\Var\left(\V^\calS\right) \simeq \SFF\Monadic^!$, but $\SFF\Monadic^! \simeq \Mnd_\sff\left(\V^\calS\right)^\op$ by \cite[\S 9.3]{Pres}.  
\end{proof}

\noindent A related result that is established in a more general setting is \cite[Theorem 5.26]{EP}.

\begin{para}[\textbf{$\N_\calS$-nervous $\V$-monads and $\N_\calS$-theories}]
\label{nervous_para}
We say that a $\V$-monad $\T$ on $\V^\calS$ is \emph{$\N_\calS$-nervous} if it satisfies the (somewhat technical) conditions of \cite[Definition 4.9]{Struct}, one of which is that the full sub-$\V$-category of $\T\Alg$ consisting of the free $\T$-algebras on objects of $\N_\calS$ is dense (in the enriched sense). Since the subcategory of arities $\N_\calS \hookrightarrow \V^\calS$ is eleutheric \eqref{eleuth_para}, we deduce from \cite[Corollary 5.1.14]{Struct} that the strongly finitary (i.e.~$\N_\calS$-ary, \ref{eleuth_para}) $\V$-monads on $\V^\calS$ coincide with the $\N_\calS$-nervous $\V$-monads on $\V^\calS$.

An \emph{$\N_\calS$-theory} \cite[Definition 3.1]{Struct} is a $\V$-category $\mathfrak{T}$ equipped with an identity-on-objects $\V$-functor $\tau : \N_\calS^\op \to \mathfrak{T}$ satisfying the condition that each $\mathfrak{T}(J, \tau-) : \N_\calS^\op \to \V$ ($J \in \N_\calS$) is a \emph{nerve} for the inclusion $j : \N_\calS \hookrightarrow \V^\calS$, meaning that $\mathfrak{T}(J, \tau-) \cong \V^\calS(j-, X)$ for some $X \in \ob\V^\calS$. We write $\Th_{\N_\calS}\left(\V^\calS\right)$ for the (ordinary) category of $\N_\calS$-theories and their morphisms (see \cite[Definition 3.1]{Struct}). A \emph{(concrete) $\frakT$-algebra} is an object $A$ of $\V^\calS$ equipped with a $\V$-functor $M : \frakT \to \V$ satisfying $M \circ \tau = \V^\calS(j-, A) : \N_\calS^\op \to \V$. There is a $\V$-category $\frakT\Alg$ of $\frakT$-algebras and a forgetful $\V$-functor $U^\frakT : \frakT\Alg \to \V^\calS$ (see \cite[Definition 3.2]{Struct}), so that $\frakT\Alg$ may be regarded as a $\V$-category over $\V^\calS$. When $\calS$ is a singleton, $\N_\calS$-theories are precisely the enriched algebraic theories studied by Borceux and Day \cite{BorceuxDay}.

Since the subcategory of arities $\N_\calS \hookrightarrow \V^\calS$ is eleutheric \eqref{eleuth_para}, it is then \emph{amenable} (in the sense of \cite[Definition 3.12]{Struct}) by \cite[Theorem 5.1.9]{Struct}. We then deduce from \cite[Theorem 4.13]{Struct} and the foregoing that there is an equivalence $\Th_{\N_\calS}\left(\V^\calS\right) \simeq \Mnd_{\sff}\left(\V^\calS\right)$ between the category of $\N_\calS$-theories and the category of strongly finitary $\V$-monads on $\V^\calS$, with the property that the $\V$-category of algebras for an $\N_\calS$-theory is isomorphic (in $\V\CAT/\V^\calS$) to the $\V$-category of algebras for the corresponding strongly finitary $\V$-monad. From Theorem \ref{variety_thm} we then immediately obtain the following:  
\end{para} 

\begin{theo}
\label{variety_thm_2}
Under the assumptions of \ref{given_data}, the category $\Var\left(\V^\calS\right)$ of $\V$-enriched $\calS$-sorted varieties is dually equivalent to the category $\Th_{\N_\calS}\left(\V^\calS\right)$ of $\N_\calS$-theories:
\[ \Var\left(\V^\calS\right) \simeq \Th_{\N_\calS}\left(\V^\calS\right)^\op. \] 
\end{theo}

\section{Examples}
\label{examples_section}

In this section, we provide some examples of our main setting \eqref{given_data}: a complete and cocomplete cartesian closed category $\V$ such that $|-| \cong \V(1, -) : \V \to \Set$ is faithful. For certain examples of such $\V$ (in \S\ref{rel_subsection} and \S\ref{CPO_subsection}), we shall also establish even more concrete, syntactic formulations of $\V$-enriched $\calS$-sorted equational theories, in which arbitrary $\V$-enriched $\calS$-sorted signatures are replaced by \emph{classical} $\calS$-sorted signatures (where all parameter objects are trivial, \ref{signature}) together with certain additional syntactic constructions related to the specific structure of $\V$ (see Definitions \ref{rel_theory} and \ref{CPO_theory}).  

\subsection{$\Set$}
\label{set_subsection}

The category $\Set$ is complete, cocomplete, and cartesian closed, and the identity functor on $\Set$ is (obviously) faithful and represented by the singleton set. So $\V = \Set$ is an example of our main setting, and from Theorem \ref{variety_thm} we recover the classical result that the category $\Var\left(\Set^\calS\right)$ of \emph{classical} (i.e.~$\Set$-enriched) $\calS$-sorted varieties is dually equivalent to the category $\Mnd_\sff\left(\Set^\calS\right)$ of (strongly) finitary monads on $\Set^\calS$ (see e.g.~\cite[Chapter 14]{Algebraic_theories}).

\subsection{Relational Horn theories}
\label{rel_subsection}

We establish in this subsection that if $\T$ is a \emph{relational Horn theory} satisfying certain conditions, then the category $\T\Mod$ of \emph{$\T$-models} and their morphisms is an example of our main setting \eqref{given_data}. We also provide an even more concrete, syntactic formulation of $(\T\Mod)$-enriched $\calS$-sorted equational theories in terms of \emph{classical $\calS$-sorted equational theories with $\T$-relations} \eqref{rel_theory}. 

We first review relational Horn theories and their categories of models; much of the content of this section is taken from \cite{Extensivity, Exp_relational}; see also \cite[\S 3]{Monadsrelational}. 

\begin{defn_sub}
\label{relational_sig}
{
A \textbf{(finitary) relational signature} is a set $\Pi$ of \emph{relation symbols} equipped with an assignment to each relation symbol $R \in \Pi$ of a natural number \emph{arity} $ \sfar(R) \geq 0$.\footnote{One can also consider more general (infinitary) relational signatures, where the arities of relation symbols are permitted to be arbitrary (possibly infinite) cardinal numbers; for example, see \cite{Rosickyconcrete}.}
}
\end{defn_sub}

\noindent We fix a relational signature $\Pi$ for the rest of \S\ref{rel_subsection} (unless otherwise stated).

\begin{defn_sub}
\label{edge}
{
A \textbf{$\Pi$-edge} in a set $S$ is a pair $\left(R, \left(s_1, \ldots, s_{\sfar(R)}\right)\right)$ consisting of a relation symbol $R \in \Pi$ and a tuple $\left(s_1, \ldots, s_{\sfar(R)}\right) \in S^{\sfar(R)}$. In other words, a $\Pi$-edge in $S$ is an element of the set-theoretic coproduct $\coprod_{R \in \Pi} S^{\sfar(R)}$. A \textbf{$\Pi$-structure} $X$ consists of a set $|X|$ equipped with a set $\sfE(X)$ of $\Pi$-edges in $|X|$. We can also describe a $\Pi$-structure $X$ as a set $|X|$ equipped with a subset $R^X \subseteq |X|^{\sfar(R)}$ for each relation symbol $R \in \Pi$: given $R \in \Pi$, we have $\left(x_1, \ldots, x_{\sfar(R)}\right) \in R^X$ iff $\sfE(X)$ contains the $\Pi$-edge $\left(R, \left(x_1, \ldots, x_{\sfar(R)}\right)\right)$. We shall often write $X \models Rx_1\ldots x_{\sfar(R)}$ instead of $\left(x_1, \ldots, x_{\sfar(R)}\right) \in R^X$ or $\left(R, \left(x_1, \ldots, x_{\sfar(R)}\right)\right) \in \sfE(X)$.
}
\end{defn_sub} 

\noindent When $R$ is a binary relation symbol (i.e.~its arity is $2$), we shall sometimes write $X \models x_1Rx_2$ rather than $X \models Rx_1x_2$.   

\begin{defn_sub}
\label{Pi_morphism}
{
Let $h : S \to T$ be a function from a set $S$ to a set $T$, and let $e = (R, (s_1, \ldots, s_n))$ be a $\Pi$-edge in $S$. We write $h \cdot e = h \cdot (R, (s_1, \ldots, s_n))$ for the $\Pi$-edge $(R, (h(s_1), \ldots, h(s_n)))$ in $T$. For a set $E$ of $\Pi$-edges in $S$, we write $h \cdot E$ for the set of $\Pi$-edges $\{h \cdot e \mid e \in S\}$ in $T$. A \textbf{$\Pi$-morphism $h : X \to Y$} from a $\Pi$-structure $X$ to a $\Pi$-structure $Y$ is a function $h : |X| \to |Y|$ such that $h \cdot \sfE(X) \subseteq \sfE(Y)$. We write $\Pi\Str$ for the category of $\Pi$-structures and their morphisms, which is a concrete category over $\Set$ by way of the forgetful functor $|-| : \Pi\Str \to \Set$ that sends each $\Pi$-structure $X$ to its underlying set $|X|$.    
}
\end{defn_sub} 

\noindent We now turn to the syntax of relational Horn theories. For the rest of \S\ref{rel_subsection}, we fix an infinite set of variables $\Var$.  

\begin{defn_sub}
\label{Horn_formula}
A \textbf{relational Horn formula (over $\Pi$)} is an expression of the form $\Phi \Longrightarrow \psi$, where $\Phi$ is a set of $\Pi$-edges in $\Var$ and $\psi$ is a $\left(\Pi \cup \{\doteq\}\right)$-edge in $\Var$, where $\doteq$ is a binary relation symbol not in $\Pi$. If $\Phi = \{\varphi_1, \ldots, \varphi_n\}$ is finite, then we write $\varphi_1, \ldots, \varphi_n \Longrightarrow \psi$, and if $\Phi = \varnothing$, then we write $\Longrightarrow \psi$. A \textbf{relational Horn formula without equality (over $\Pi$)} is a relational Horn formula $\Phi \Longrightarrow \psi$ (over $\Pi$) such that $\psi$ is a $\Pi$-edge in $\Var$ (i.e.~$\psi$ does not contain the relation symbol $\doteq$).  
\end{defn_sub}

\noindent We shall typically write $\Pi$-edges in $\Var$ as $Rv_1 \ldots v_n$ rather than $(R, (v_1, \ldots, v_n))$, and when $R \in \Pi$ has arity $2$, we shall typically write $v_1 R v_2$ rather than $Rv_1v_2$.

\begin{defn_sub}
\label{Horn_theory}
{
A \textbf{relational Horn theory $\T$ (without equality)} is a set of relational Horn formulas (without equality) over $\Pi$, which we call the \emph{axioms} of $\T$.   
}
\end{defn_sub}

\begin{defn_sub}
\label{formula_satis}
Let $X$ be a $\Pi$-structure. We define a $\left(\Pi \cup \{\doteq\}\right)$-structure $\overline{X}$ by $\left|\overline{X}\right| := |X|$ and $\sfE\left(\overline{X}\right) := \sfE(X) \cup \{(\doteq, (x, x)) \mid x \in |X|\}$. A \textbf{valuation in $X$} is a function $\kappa : \Var \to |X|$. We say that $X$ \textbf{satisfies} a relational Horn formula $\Phi \Longrightarrow \psi$ if $\overline{X} \models \kappa \cdot \psi$ for every valuation $\kappa$ in $X$ such that $X \models \kappa \cdot \varphi$ for all $\varphi \in \Phi$. A \textbf{model} of a relational Horn theory $\T$, or a \textbf{$\T$-model}, is a $\Pi$-structure that satisfies each axiom of $\T$. We write $\T\Mod$ for the full subcategory of $\Pi\Str$ consisting of the $\T$-models, so that $\T\Mod$ can be regarded as a concrete category over $\Set$ by restricting the forgetful functor $|-| : \Pi\Str \to \Set$.           
\end{defn_sub}

\begin{para_sub}[\textbf{Conditions for $\T\Mod$ to be an example of our main setting}]
\label{rel_ter_para}
Let $\T$ be a relational Horn theory. The category $\T\Mod$ is complete and cocomplete, and even \emph{locally presentable} (by \cite[Proposition 5.30]{LPAC}). The full subcategory $\T\Mod \hookrightarrow \Pi\Str$ is reflective by \cite[Proposition 3.6]{Monadsrelational}, so that every $\Pi$-structure generates a free $\T$-model. 

The terminal object of $\T\Mod$ is the $\Pi$-structure $1$ with $|1| = \{\ast\}$ and $1 \models R\ast\ldots\ast$ for all $R \in \Pi$ (i.e.~it is the \emph{indiscrete} $\Pi$-structure on the singleton set $\{\ast\}$). Given a set $S$, we say that a relation $R$ on $S$ is \emph{reflexive} if $(s, \ldots, s) \in R$ for all $s \in S$. The forgetful functor $|-| : \T\Mod \to \Set$ is then represented by the terminal object $1$ iff $\T$ is \emph{reflexive}, in the sense that for every $\T$-model $X$ and every $R \in \Pi$ the relation $R^X$ on $|X|$ is reflexive.

Thus, given a relational Horn theory $\T$, we deduce that the category $\T\Mod$ provides an example of our main setting if $\T$ is reflexive and $\T\Mod$ is cartesian closed.
\end{para_sub}

\begin{ex_sub}
\label{rel_examples}
{
We provide the following examples of relational Horn theories $\T$ such that $\T$ is reflexive and $\T\Mod$ is cartesian closed, so that $\T\Mod$ is an example of our main setting by \ref{rel_ter_para}.
\begin{enumerate}[leftmargin=*]
\item Given an arbitrary relational signature $\Pi$, suppose that $\T$ is the relational Horn theory without equality that simply consists of the axioms $\Longrightarrow Rv\ldots v$ for all $R \in \Pi$, so that a $\T$-model is just a $\Pi$-structure in which each relation is reflexive. Then $\T$ is reflexive and $\T\Mod$ is cartesian closed by \cite[Theorem 6.15]{Exp_relational}.

\item Let $\Pi$ consist of a single binary relation symbol $\leq$, and let $\T$ be the reflexive relational Horn theory without equality over $\Pi$ whose axioms are $\Longrightarrow v \leq v$ and $v_1 \leq v_2, v_2 \leq v_3 \Longrightarrow v_1 \leq v_3$. Then $\T\Mod$ is the cartesian closed category $\mathsf{Preord}$ of preordered sets and monotone functions. If one adds the additional axiom $v_1 \leq v_2, v_2 \leq v_1 \Longrightarrow v_1 \doteq v_2$, then the category of models of the resulting reflexive relational Horn theory is the cartesian closed category $\mathsf{Pos}$ of posets (partially ordered sets) and monotone functions. 

\item The following examples originate from \cite{Kelly, Metagories}. Let $(\Q, \leq)$ be a \emph{complete Heyting algebra}, i.e.~a complete lattice in which each $v \wedge (-) : \Q \to \Q$ ($v \in \Q$) preserves arbitrary suprema. A \emph{$\Q$-graph} or \emph{$\Q$-valued relation} $(X, d)$ is a set $X$ equipped with a function $d : X \times X \to \Q$. A \emph{reflexive $\Q$-graph} is a $\Q$-graph $(X, d)$ satisfying $d(x, x) = \top$ for all $x \in X$. A \emph{$\Q$-category} is a reflexive $\Q$-graph $(X, d)$ satisfying $d(x, z) \geq d(x, y) \wedge d(y, z)$ for all $x, y, z \in X$. A \emph{pseudo-$\Q$-metric space} is a $\Q$-category $(X, d)$ satisfying $d(x, y) = d(y, x)$ for all $x, y \in X$. Finally, a \emph{$\Q$-metric space} is a pseudo-$\Q$-metric space $(X, d)$ satisfying $d(x, y) = \top \Longrightarrow x = y$ for all $x, y \in X$. If $(X, d_X)$ and $(Y, d_Y)$ are $\Q$-graphs, then a \emph{$\Q$-functor} or \emph{$\Q$-contraction} $f : (X, d_X) \to (Y, d_Y)$ is a function $f : X \to Y$ such that $d_X(x, x') \leq d_Y(f(x), f(x'))$ for all $x, x' \in X$. We write $\Q\text{-}\Gph$ for the concrete category of $\Q$-graphs and $\Q$-functors, and we write $\Q\RGph$ (resp.~$\Q\Cat$, $\PMet_\Q$, $\Met_\Q$) for the full subcategory of $\Q\text{-}\Gph$ consisting of the reflexive $\Q$-graphs (resp.~the $\Q$-categories, the pseudo-$\Q$-metric spaces, the $\Q$-metric spaces).

Let $\Pi_\Q$ be the relational signature that consists of binary relation symbols $\sim_q$ for all $q \in \Q$. We write $\T_{\Q\text{-}\Gph}$ for the relational Horn theory over $\Pi_\Q$ that consists of the axioms $v_1 \sim_q v_2 \Longrightarrow v_1 \sim_{q'} v_2$ for all $q, q' \in \Q$ with $q \geq q'$, together with the axioms $\{v_1 \sim_{q_i} v_2 \mid i \in I\} \Longrightarrow v_1 \sim_{\bigvee_i q_i} v_2$ for all small families $\left(q_i\right)_{i \in I}$ of elements of $\Q$. We write $\T_{\Q\RGph}$ for the relational Horn theory over $\Pi_\Q$ that extends $\T_{\Q\text{-}\Gph}$ by adding the single axiom $\Longrightarrow v \sim_\top v$. We write $\T_{\Q\Cat}$ for the relational Horn theory over $\Pi_\Q$ that extends $\T_{\Q\RGph}$ by adding the axioms $v_1 \sim_q v_2, v_2 \sim_{q'} v_3 \Longrightarrow v_1 \sim_{q \wedge q'} v_3$ for all $q, q' \in \Q$. We let $\T_{\PMet_\Q}$ be the relational Horn theory over $\Pi_\Q$ that extends $\T_{\Q\Cat}$ by adding the axioms $v_1 \sim_q v_2 \Longrightarrow v_2 \sim_q v_1$ for all $q \in \Q$. Finally, we let $\T_{\Met_\Q}$ be the relational Horn theory over $\Pi_\Q$ that extends $\T_{\PMet_\Q}$ by adding the single axiom $v_1 \sim_\top v_2 \Longrightarrow v_1 \doteq v_2$. It is shown in \cite[Appendix]{Extensivity} that $\T_{\Q\RGph}\Mod$ (resp.~$\T_{\Q\Cat}\Mod$, $\T_{\PMet_\Q}\Mod$, $\T_{\Met_\Q}\Mod$) is isomorphic in $\mathsf{CAT}/\Set$ to $\Q\RGph$ (resp.~$\Q\Cat$, $\PMet_\Q$, $\Met_\Q$). 

Given $\T \in \left\{\T_{\Q\RGph}, \T_{\Q\Cat}, \T_{\PMet_\Q}, \T_{\Met_\Q}\right\}$, the relational Horn theory $\T$ is readily seen to be reflexive. Moreover, the category $\Q\RGph \cong \T_{\Q\RGph}\Mod$ is cartesian closed by \cite[Theorem 7.17]{Exp_relational}, while the remaining categories are cartesian closed by \cite[Theorem 6.20]{Exp_relational}. 

Note that when $\Q$ is the complete Heyting algebra $([0, \infty], \geq)$ (the extended real half-line with $\geq$) with binary meets given by maxima, then $\Met_\Q = \mathsf{Ult}\Met$, the category of \emph{(extended) ultrametric spaces} (allowing infinite distances).

\item The following example originates from \cite{Dubucquasitopoi} (see also \cite[Page 331]{Rosickyconcrete}). A category $\bbC$ is \emph{well-pointed} if it has a terminal object $1$ such that $|-| \cong \bbC(1, -) : \bbC \to \Set$ is faithful, thus exhibiting $\bbC$ as a concrete category over $\Set$. A \emph{finitary concrete site} $(\bbC, \calJ)$ is a well-pointed category $\bbC$ such that for each object $C$ of $\bbC$, the set $|C|$ is finite, together with a Grothendieck topology $\calJ$ on $\bbC$ whose covering families are all \emph{jointly surjective}. Given a finitary concrete site $(\bbC, \calJ)$ and a set $S$, a \emph{plot (in $S$)}\footnote{This terminology comes from \cite{Baezsmooth}.} is a function $p : |C| \to S$ for some object $C$ of $\bbC$. A class $\mathscr{P}$ of plots in $S$ is \emph{admissible} if it satisfies the following conditions:
\begin{itemize}[leftmargin=*]
\item $\scrP$ contains all constant plots; i.e.~for each object $C$ of $\bbC$, each constant plot $|C| \to S$ is in $\scrP$. 
\item $\scrP$ is closed under precomposition with morphisms of $\bbC$; i.e.~for each morphism $h : C \to D$ of $\bbC$ and each plot $p : |D| \to S$ in $\scrP$, the plot $p \circ h : |C| \to S$ is in $\scrP$.  
\item $\scrP$ satisfies the \emph{sheaf} (or \emph{gluing}) \emph{condition}: for each object $C$ of $\bbC$ and each covering family $\left(h_j : C_j \to C\right)_{j \in J} \in \calJ(C)$, a plot $p : |C| \to S$ is in $\scrP$ if each plot $p \circ h_j : \left|C_j\right| \to S$ ($j \in J$) is in $\scrP$. 
\end{itemize}
A \emph{quasispace $X$} (\emph{on} $(\bbC, \calJ)$) is a set $|X|$ equipped with an admissible class of plots $\scrP_X$ in $|X|$. 
Given quasispaces $X$ and $Y$, a \emph{morphism of quasispaces} $f : X \to Y$ is a function $f : |X| \to |Y|$ that sends plots to plots, i.e.~for each object $C$ of $\bbC$ and each plot $p : |C| \to |X|$ in $\scrP_X$, the plot $f \circ p : |C| \to |Y|$ is in $\scrP_Y$. We write $\QuasiSp(\bbC, \calJ)$ for the category of quasispaces on $(\bbC, \calJ)$ and their morphisms, which is equipped with the forgetful functor $|-| : \QuasiSp(\bbC, \calJ) \to \Set$ that sends a quasispace to its underlying set. The category $\QuasiSp(\bbC, \calJ)$ is cartesian closed. 

Given a finitary concrete site $(\bbC, \calJ)$, we write $\Pi_{(\bbC, \calJ)}$ for the relational signature consisting of a relation symbol $R_C$ of arity $\#|C|$ (the cardinality of the finite set $|C|$) for each $C \in \ob\bbC$. We then write $\T_{(\bbC, \calJ)}$ for the reflexive relational Horn theory without equality over $\Pi_{(\bbC, \calJ)}$ that consists of the following axioms:
\begin{itemize}[leftmargin=*]
\item For each object $C$ of $\bbC$, the axiom \[ \Longrightarrow R_C v \ldots v. \]
\item For each morphism $h : C \to D$ of $\bbC$ with $|C| = \{c_1, \ldots, c_n\}$ and $|D| = \{d_1, \ldots, d_m\}$, the axiom \[ R_D v_{d_1} \ldots v_{d_m} \Longrightarrow R_C v_{h(c_1)} \ldots v_{h(c_n)}, \] where the variables $v_{d_1} \ldots v_{d_m}$ are pairwise distinct.
\item For each object $C$ of $\bbC$ with $|C| = \{c_1, \ldots, c_n\}$ and each covering family $(h_j : C_j \to C)_{j \in J} \in \calJ(C)$ with $\left|C_j\right| = \{c_{j, 1}, \ldots, c_{j, n_j}\}$ for each $j \in J$, the axiom
\[ \left\{ R_{C} v_{h_j\left(c_{j, 1}\right)} \ldots v_{h_j\left(c_{j, n_j}\right)} \mid j \in J\right\} \Longrightarrow R_C v_{c_1} \ldots v_{c_n}, \] where the variables $v_{c_1} \ldots v_{c_n}$ are pairwise distinct.    
\end{itemize}
It is then straightforward to verify that $\T_{(\bbC, \calJ)}\Mod \cong \QuasiSp(\bbC, \calJ)$ as concrete categories over $\Set$. In detail, given a model $X$ of $\T_{(\bbC, \calJ)}$, one defines a corresponding quasispace $\widehat{X}$ with $\left|\widehat{X}\right| := |X|$ and \[ \scrP_{\widehat{X}} := \bigcup_{C \in \ob\bbC} \left\{ p : |C| = \{c_1, \ldots, c_n\} \to |X| \mid X \models R_C p(c_1) \ldots p(c_n)\right\}. \] Conversely, given a quasispace $X$, one defines a corresponding model $\widetilde{X}$ of $\T_{(\bbC, \calJ)}$ with $\left|\widetilde{X}\right| := |X|$ and 
\[ \sfE\left(\widetilde{X}\right) := \bigcup_{C \in \ob\bbC, |C| = \{c_1, \ldots, c_n\}} \left\{\left(R_C, \left(p(c_1), \ldots, p(c_n)\right)\right) \mid p : |C| \to |X| \text{ in } \scrP_X\right\}. \]
Since $\T_{(\bbC, \calJ)}$ is reflexive and $\T_{(\bbC, \calJ)}\Mod \cong \QuasiSp(\bbC, \calJ)$ is cartesian closed, we deduce that $\QuasiSp(\bbC, \calJ)$ is an example of our main setting for every finitary concrete site $(\bbC, \calJ)$.  

\item As a special case of the preceding example, the category $\Simp$ of \emph{(abstract) simplicial complexes} is an example of our main setting, because $\Simp$ is equivalent to the category of quasispaces on the finitary concrete site $\left(\mathsf{FinCard}^+, \mathsf{Triv}\right)$ consisting of the well-pointed category $\mathsf{FinCard}^+$ of positive finite cardinals (and all functions between them) equipped with the trivial Grothendieck topology $\mathsf{Triv}$ (see e.g.~\cite[Proposition 4.18]{Baezsmooth}). However, the concrete category $\Simp$ is usually presented in the following way. Given a set $S$, a \emph{simplex (in $S$)} is a non-empty finite subset of $S$. A set $\mathscr{S}$ of simplices in $S$ is \emph{admissible} if it contains all singleton simplices and is downward closed (i.e.~if $s, s'$ are simplices in $S$ with $s' \subseteq s$ and $s \in \scrS$, then $s' \in \scrS$). A(n) \emph{(abstract) simplicial complex} $X$ is a set $|X|$ equipped with an admissible set of simplices $\scrS_X$ in $|X|$. Given simplicial complexes $X$ and $Y$, a \emph{morphism of simplicial complexes $f : X \to Y$} is a function $f : |X| \to |Y|$ that sends simplices to simplices, i.e.~for each $s \in \scrS_X$ we have $f[s] \in \scrS_Y$. We write $\Simp$ for the category of simplicial complexes and their morphisms, which is equipped with the forgetful functor $|-| : \Simp \to \Set$ that sends a simplicial complex to its underlying set. 

The corresponding relational signature $\Pi_\Simp = \Pi_{\left(\mathsf{FinCard}^+, \mathsf{Triv}\right)}$ then consists of one relation symbol $R_n$ of arity $n$ for each natural number $n \geq 1$, and the corresponding relational Horn theory $\T_\Simp = \T_{\left(\mathsf{FinCard}^+, \mathsf{Triv}\right)}$ without equality over $\Pi_\Simp$ consists of the following (simplified) axioms:
\begin{itemize}[leftmargin=*]
\item For each $n \geq 1$, the axiom \[ \Longrightarrow R_n v \ldots v. \]

\item For all $n, m \geq 1$ and each function $h : \{1, \ldots, n\} \to \{1, \ldots, m\}$, the axiom \[ R_m v_{1} \ldots v_{m} \Longrightarrow R_n v_{h(1)} \ldots v_{h(n)}, \] where the variables $v_{1} \ldots v_{m}$ are pairwise distinct.
\end{itemize} 
It is then straightforward to verify that $\T_\Simp\Mod \cong \Simp$ as concrete categories over $\Set$. In detail, given a model $X$ of $\T_\Simp$, one defines a corresponding simplicial complex $\widehat{X}$ on the same underlying set $|X|$ with 
\[ \scrS_{\widehat{X}} := \bigcup_{n \geq 1} \left\{\{x_1, \ldots, x_n\} \subseteq |X| \mid X \models R_n x_1\ldots x_n\right\}. \] Conversely, given a simplicial complex $X$, one defines a corresponding model $\widetilde{X}$ of $\T_\Simp$ on the same underlying set $|X|$ with $\widetilde{X} \models R_n x_1 \ldots x_n$ ($n \geq 1, x_1, \ldots, x_n \in |X|$) iff $\{x_1, \ldots, x_n\} \in \scrS_X$. 
\end{enumerate}
}
\end{ex_sub}

\noindent For the remainder of \S\ref{rel_subsection} we fix a reflexive relational Horn theory $\T$ for which $\T\Mod$ is cartesian closed, so that $\T\Mod$ is an example of our main setting by \ref{rel_ter_para}. Our objective is now to show (in Propositions \ref{rel_prop_1} and \ref{rel_prop_2}) that ($\T\Mod$)-enriched $\calS$-sorted equational theories admit an even more concrete, syntactic formulation in terms of \emph{classical $\calS$-sorted equational theories with $\T$-relations} \eqref{rel_theory}. 

\begin{defn_sub}
\label{sig_pi_relation}
Let $\Sigma$ be a classical $\calS$-sorted signature \eqref{signature}. Given an $\calS$-sorted variable context $\vbar$ and a sort $S \in \calS$, a \textbf{$(\Sigma, \Pi)$-relation of sort $S$ in context $\vbar$} is a $\Pi$-edge in the set $\Term(\Sigma; \vbar)_S$ of $\Sigma$-terms of sort $S$ in context $\vbar$ \eqref{syntax_para}. We write a $(\Sigma, \Pi)$-relation of sort $S$ in context $\vbar$ as \[ \left[\vbar \vdash Rt_1\ldots t_n : S\right] \] rather than as $\left(R, \left([\vbar \vdash t_1 : S], \ldots, [\vbar \vdash t_n : S]\right)\right)$. 

A \textbf{$(\Sigma, \Pi)$-relation in context $\vbar$} is a $(\Sigma, \Pi)$-relation of sort $S$ in context $\vbar$ for some sort $S \in \calS$, while a \textbf{$(\Sigma, \Pi)$-relation (in context)} is a $(\Sigma, \Pi)$-relation in context $\vbar$ for some context $\vbar$. 
\end{defn_sub}

\noindent Recall from \ref{syntax_para} that given a classical $\calS$-sorted signature $\Sigma$, a $\Sigma$-algebra $A$, $J \in \N_\calS$, and a $\Sigma$-term $\left[\vbar : J \vdash t : S\right]$ of sort $S$ in context $\vec{v} : J$ \eqref{notation_para}, we write $t^A : A^J \to A_S$ for the $\V$-morphism $\left[\vbar : J \vdash t : S\right]^A$.  

\begin{defn_sub}
\label{sig_pi_relation_sat}
Let $\Sigma$ be a classical $\calS$-sorted signature \eqref{signature}, let $A$ be a $\Sigma$-algebra, let $J \in \N_\calS$, and let $\left[\vbar : J \vdash Rt_1\ldots t_n : S\right]$ be a $(\Sigma, \Pi)$-relation in context $\vbar : J$ \eqref{notation_para}. We say that $A$ \textbf{satisfies} the $(\Sigma, \Pi)$-relation $\left[\vbar : J \vdash Rt_1\ldots t_n : S\right]$ if \[ \T\Mod\left(A^J, A_S\right) \models Rt_1^A \ldots t_n^A, \]
recalling that $\T\Mod\left(A^J, A_S\right)$ is a $\T$-model (as an internal hom of the cartesian closed category $\T\Mod$).    
\end{defn_sub}

\begin{defn_sub}
\label{rel_theory}
A \textbf{classical $\calS$-sorted equational theory with $\T$-relations} is a triple $\calP = (\Sigma, \calR, \calE)$ consisting of a classical $\calS$-sorted signature $\Sigma$ \eqref{signature}, a set $\calR$ of $(\Sigma, \Pi)$-relations in context \eqref{sig_pi_relation}, and a set $\calE$ of syntactic $\Sigma$-equations in context \eqref{syntax_para}. A \textbf{$\calP$-algebra} is a $\Sigma$-algebra that satisfies each $(\Sigma, \Pi)$-relation in $\calR$ \eqref{sig_pi_relation_sat} and each syntactic $\Sigma$-equation in $\calE$ \eqref{sat_synt}. We write $\calP\Alg \hookrightarrow \Sigma\Alg$ for the full sub-($\T\Mod$)-category consisting of the $\calP$-algebras, which is equipped with a faithful $(\T\Mod)$-functor $U^\calP : \calP\Alg \to \T\Mod^\calS$ obtained by restricting $U^\Sigma : \Sigma\Alg \to \T\Mod^\calS$. When $\calS$ is a singleton, we instead say \textbf{classical single-sorted equational theory with $\T$-relations}.
\end{defn_sub}

\noindent In the next two results, we show that ($\T\Mod$)-enriched $\calS$-sorted equational theories have the same ``expressive power'' as classical $\calS$-sorted equational theories with $\T$-relations. 

\begin{prop_sub}
\label{rel_prop_1}
For every ($\T\Mod$)-enriched $\calS$-sorted equational theory $\calT = (\Sigma, \calE)$, there is a classical $\calS$-sorted equational theory with $\T$-relations $\calP_\calT$ such that $\calT\Alg \cong \calP_\calT\Alg$ in $(\T\Mod)\CAT/\T\Mod^\calS$.
\end{prop_sub}

\begin{proof}
The classical $\calS$-sorted signature for $\calP_\calT$ is $|\Sigma|$, the underlying classical $\calS$-sorted signature of $\Sigma$ \eqref{ord_sig}. To define the set $\calR$ of $(|\Sigma|, \Pi)$-relations for $\calP_\calT$, we consider for each operation symbol $\sigma \in \Sigma$ of type $(J, S, P)$ and each $\Pi$-edge $P \models Rp_1\ldots p_n$ of $P$ ($R \in \Pi$) the $(|\Sigma|, \Pi)$-relation $\left[\vbar : J \vdash R\sigma_{p_1}\left(\vbar\right)\ldots\sigma_{p_n}\left(\vbar\right) : S\right]$ of sort $S$ in context $\vbar : J$ \eqref{notation_para}. We then define $\calR$ to consist of precisely all such $(|\Sigma|, \Pi)$-relations. We set $\calP_\calT := \left(|\Sigma|, \calR, \calE\right)$, using the same set $\calE$ of syntactic $|\Sigma|$-equations from $\calT$. 

We now show that $\calT\Alg \cong \calP_\calT\Alg$ in $(\T\Mod)\CAT/\T\Mod^\calS$. Given a $\calT$-algebra $A$, we have the underlying $|\Sigma|$-algebra $|A|$ \eqref{ord_sig_para} with the same carrier object $A$, and we show that $|A|$ is a $\calP_\calT$-algebra. The $|\Sigma|$-algebra $|A|$ already satisfies all of the syntactic $|\Sigma|$-equations in $\calE$, because the $\Sigma$-algebra $A$ is a $\calT$-algebra. Now let $\sigma \in \Sigma$ be of type $(J, S, P)$, let $P \models Rp_1\ldots p_n$ be a $\Pi$-edge of $P$, and let us show that $|A|$ satisfies the $(|\Sigma|, \Pi)$-relation $\left[\vbar : J \vdash R\sigma_{p_1}\left(\vbar\right)\ldots\sigma_{p_n}\left(\vbar\right) : S\right]$ of sort $S$ in context $\vbar : J$. So we must show that 
\[ \T\Mod\left(A^J, A_S\right) \models R\sigma^A\left(p_1\right)\ldots\sigma^A\left(p_n\right), \] and this is true because $P \models Rp_1\ldots p_n$ and $\sigma^A : P \to \T\Mod\left(A^J, A_S\right)$ is a $\Pi$-morphism. Thus, the $|\Sigma|$-algebra $|A|$ is a $\calP_\calT$-algebra. 

Conversely, let $A$ be a $\calP_\calT$-algebra, and let us define a corresponding $\calT$-algebra $A$ with the same carrier object. For each $\sigma \in \Sigma$ of type $(J, S, P)$, we define
\[ \sigma^A : P \to \T\Mod\left(A^J, A_S\right) \]
\[ p \mapsto \sigma_p^A, \] and $\sigma^A$ is a $\Pi$-morphism because the $|\Sigma|$-algebra $A$ satisfies each $(|\Sigma|, \Pi)$-relation in $\calR$. This gives the carrier object $A$ the structure of a $\Sigma$-algebra, which is a $\calT$-algebra because the underlying $|\Sigma|$-algebra is the given $\calP_\calT$-algebra $A$.

Now recall from \ref{ord_sig_para} that we have a fully faithful $(\T\Mod)$-functor $|-|^\Sigma : \Sigma\Alg \to |\Sigma|\Alg$ with $U^{|\Sigma|} \circ |-|^\Sigma = U^\Sigma$. The foregoing establishes that this $(\T\Mod$)-functor restricts and corestricts to a fully faithful $(\T\Mod)$-functor $\calT\Alg \to \calP_\calT\Alg$ that is bijective-on-objects and thus an isomorphism in $(\T\Mod)\CAT/\T\Mod^\calS$, as desired.   
\end{proof}

\begin{prop_sub}
\label{rel_prop_2}
For every classical $\calS$-sorted equational theory with $\T$-relations $\calP$, there is a ($\T\Mod$)-enriched $\calS$-sorted equational theory $\calT_\calP$ such that $\calP\Alg \cong \calT_\calP\Alg$ in $(\T\Mod)\CAT/\T\Mod^\calS$.
\end{prop_sub}

\begin{proof}
Let $\calP = (\Sigma, \calR, \calE)$. For each pair $(J, S) \in \N_\calS \times \calS$, we consider all the $(\Sigma, \Pi)$-relations $\left[\vbar : J \vdash Rt_1\ldots t_n : S\right]$ in $\calR$ of sort $S$ in context $\vbar : J$ \eqref{notation_para}. We first write $\Term(\Sigma; \vbar : J)_S^\calR$ for the subset of $\Term(\Sigma; \vbar : J)_S$ (the set of $\Sigma$-terms of sort $S$ in context $\vbar : J$) that consists of precisely those $\Sigma$-terms of sort $S$ in context $\vec{v} : J$ that occur in some $(\Sigma, \Pi)$-relation in $\calR$. We then write $P_{J, S}$ for the $\Pi$-structure that we obtain by equipping this set $\Term(\Sigma; \vbar : J)_S^\calR$ with precisely the $\Pi$-edges corresponding to these $(\Sigma, \Pi)$-relations $\left[\vbar : J \vdash Rt_1\ldots t_n : S\right]$ from $\calR$. We then write $P_{J, S}^\sharp$ for the free $\T$-model \eqref{rel_ter_para} on the $\Pi$-structure $P_{J, S}$, so that we have a unit $\Pi$-morphism $\eta_{J, S} : P_{J, S} \to P_{J, S}^\sharp$. We then augment the classical $\calS$-sorted signature $\Sigma$ to obtain a ($\T\Mod$)-enriched $\calS$-sorted signature $\Sigma^+$ by adjoining, for each pair $\left(J, S\right) \in \N_\calS \times \calS$, an operation symbol $\sigma_{J, S}$ of type $\left(J, S, P_{J, S}^\sharp\right)$. We also augment the set $\calE$ of syntactic $\Sigma$-equations to obtain a set $\calE^+$ of syntactic $\left|\Sigma^+\right|$-equations by adjoining, for each pair $\left(J, S\right) \in \N_\calS \times \calS$, each $(\Sigma, \Pi)$-relation $\left[\vbar : J \vdash Rt_1\ldots t_n : S\right]$ in $\calR$ of sort $S$ in context $\vbar : J$, and each $1 \leq i \leq n$, the following syntactic $\left|\Sigma^+\right|$-equation of sort $S$ in context $\vbar : J$:
\begin{equation}
\label{T_P_eq}
\left[\vbar : J \vdash \left(\sigma_{J, S}\right)_{\eta_{J, S}\left(t_i\right)}\left(\vbar\right) \doteq t_i : S\right],
\end{equation} 
recalling that $\eta_{J, S} : P_{J, S} \to P_{J, S}^\sharp$ and that $\left|P_{J, S}\right| = \Term(\Sigma; \vbar : J)_S^\calR$. We then define $\calT_\calP := \left(\Sigma^+, \calE^+\right)$. Note that for each $\calT_\calP$-algebra $A$ and each $(J, S) \in \N_\calS \times \calS$, the $\Pi$-morphism
\[ \sigma_{J, S}^A \circ \eta_{J, S} : P_{J, S} \to \T\Mod\left(A^J, A_S\right) \] is given by the rule $\left[\vec{v} : J \vdash t : S\right] \mapsto t^A$ for each $\left[\vec{v} : J \vdash t : S\right] \in \Term(\Sigma; \vbar : J)_S^\calR$, because the $\calT_\calP$-algebra $A$ satisfies the syntactic equations \eqref{T_P_eq}.  

We now show that $\calP\Alg \cong \calT_\calP\Alg$ in $(\T\Mod)\CAT/\T\Mod^\calS$. Given a $\calP$-algebra $A$, we define a corresponding $\calT_\calP$-algebra $A^+$ with the same carrier object. For each $\sigma \in \Sigma \subseteq \Sigma^+$, we set $\sigma^{A^+} := \sigma^A$. For each pair $\left(J, S\right) \in \N_\calS \times \calS$, we define
\[ \sigma_{J, S}^{A^+} : P_{J, S}^\sharp \to \T\Mod\left(A^J, A_S\right) \] to be the unique morphism of $\T$-models that is induced by the $\Pi$-morphism $P_{J, S} \to \T\Mod\left(A^J, A_S\right)$ given by 
\[ [\vbar : J \vdash t : S] \mapsto t^A, \] which is a $\Pi$-morphism because the $\calP$-algebra $A$ satisfies each $(\Sigma, \Pi)$-relation in $\calR$. This equips $A$ with the structure of a $\Sigma^+$-algebra $A^+$, which satisfies every syntactic $\left|\Sigma^+\right|$-equation in $\calE \subseteq \calE^+$ because the $\Sigma$-algebra $A$ is a $\calP$-algebra. The $\Sigma^+$-algebra $A^+$ satisfies each syntactic $\left|\Sigma^+\right|$-equation in $\calE^+$ of the form \eqref{T_P_eq}, because its satisfaction is equivalent to the trivial equality $t_i^A = t_i^A : A^J \to A_S$. Thus, the $\Sigma^+$-algebra $A^+$ is a $\calT_\calP$-algebra.   

Conversely, let $A$ be a $\calT_\calP$-algebra, and let us define a corresponding $\calP$-algebra $A^-$ with the same carrier object. Since $\Sigma \subseteq \Sigma^+$, the carrier object $A$ already carries the structure of a $\Sigma$-algebra $A^-$. Furthermore, since $\calE \subseteq \calE^+$, the $\Sigma$-algebra $A^-$ satisfies each syntactic $\Sigma$-equation in $\calE$. So it remains to show that the $\Sigma$-algebra $A^-$ satisfies each $(\Sigma, \Pi)$-relation of $\calR$. For each pair $\left(J, S\right) \in \N_\calS \times \calS$, let $\left[\vbar : J \vdash Rt_1\ldots t_n : S\right]$ be a $(\Sigma, \Pi)$-relation in $\calR$ of sort $S$ in context $\vbar : J$. To show that the $\Sigma$-algebra $A^-$ satisfies this $(\Sigma, \Pi)$-relation, we must show that
\[ \T\Mod\left(A^J, A_S\right) \models Rt_1^A\ldots t_n^A. \] Since $A$ is a $\calT_\calP$-algebra, we have the $\Pi$-morphism
\[ \sigma_{J, S}^A : P_{J, S}^\sharp \to \T\Mod\left(A^J, A_S\right). \] Since $P_{J, S}^\sharp \models R\eta_{J, S}\left(t_1\right)\ldots \eta_{J, S}\left(t_n\right)$ (because $P_{J, S} \models Rt_1\ldots t_n$ and $\eta_{J, S} : P_{J, S} \to P_{J, S}^\sharp$ is a $\Pi$-morphism), we then deduce that
\[ \T\Mod\left(A^J, A_S\right) \models R\sigma_{J, S}^A\left(\eta_{J, S}\left(t_1\right)\right)\ldots \sigma_{J, S}^A\left(\eta_{J, S}\left(t_n\right)\right). \] Since $A$ is a $\calT_\calP$-algebra and therefore satisfies each syntactic equation \eqref{T_P_eq}, we have $\sigma_{J, S}^A\left(\eta_{J, S}\left(t_i\right)\right) = t_i^A$ for each $1 \leq i \leq n$, which then yields $\T\Mod\left(A^J, A_S\right) \models Rt_1^A\ldots t_n^A$, as desired. We have thus established a bijective correspondence between $\calP$-algebras and $\calT_\calP$-algebras. Indeed, we clearly have $\left(A^+\right)^- = A$ for every $\calP$-algebra $A$. And to show that $\left(A^-\right)^+ = A$ for every $\calT_\calP$-algebra $A$, it suffices to show for each pair $(J, S) \in \N_\calS \times \calS$ that $\sigma_{J, S}^{\left(A^-\right)^+} = \sigma_{J, S}^A : P_{J, S}^\sharp \to \T\Mod\left(A^J, A_S\right)$, for which it in turn suffices to show that $\sigma_{J, S}^{\left(A^-\right)^+} \circ \eta_{J, S} = \sigma_{J, S}^A \circ \eta_{J, S} : P_{J, S} \to \T\Mod\left(A^J, A_S\right)$, which is true because both functions are given by the rule $\left[\vec{v} : J \vdash t : S\right] \mapsto t^A$.

Since $\Sigma \subseteq \Sigma^+$, it is straightforward to define a canonical strongly faithful\footnote{In the sense that each structural morphism $\Sigma^+\Alg(A, B) \to \Sigma\Alg\left(A^-, B^-\right)$ ($A, B \in \Sigma^+\Alg$) is a strong monomorphism.} $(\T\Mod)$-functor $\Sigma^+\Alg \to \Sigma\Alg$ in $(\T\Mod)\CAT/\T\Mod^\calS$ that sends a $\Sigma^+$-algebra $A$ to the $\Sigma$-algebra $A^-$ defined above. The foregoing then establishes that this $(\T\Mod)$-functor restricts and corestricts to a strongly faithful $(\T\Mod)$-functor $\calT_\calP\Alg \to \calP\Alg$ in $(\T\Mod)\CAT/\T\Mod^\calS$ that is bijective-on-objects. To show that each structural morphism $\calT_\calP\Alg(A, B) \rightarrowtail \calP\Alg\left(A^-, B^-\right)$ $(A, B \in \calT_\calP\Alg$) is an isomorphism, it then suffices to show that it is surjective and thus epimorphic (since $|-| : \V \to \Set$ is faithful). So let $A$ and $B$ be $\calT_\calP$-algebras, let $f : A^- \to B^-$ be a morphism of the corresponding $\calP$-algebras, and let us show that $f : A \to B$ is also a morphism of $\calT_\calP$-algebras, i.e.~of $\Sigma^+$-algebras. Since $f$ is already a morphism of $\Sigma$-algebras, it remains to show for each $(J, S) \in \N_\calS \times \calS$ that the following diagram commutes: 
\[\begin{tikzcd}
	{P_{J, S}^\sharp} &&& {\T\Mod\left(B^J, B_S\right)} \\
	\\
	{\T\Mod\left(A^J, A_S\right)} &&& {\T\Mod\left(A^J, B_S\right)},
	\arrow["{\sigma_{J, S}^B}", from=1-1, to=1-4]
	\arrow["{\sigma_{J, S}^A}"', from=1-1, to=3-1]
	\arrow["{\T\Mod\left(f^J, 1\right)}", from=1-4, to=3-4]
	\arrow["{\T\Mod\left(1, f_S\right)}"', from=3-1, to=3-4]
\end{tikzcd}\]
for which it suffices to show that the two composites are equal when precomposed with $\eta_{J, S} : P_{J, S} \to P^\sharp_{J, S}$. It therefore suffices to show for each $\Sigma$-term $[\vbar : J \vdash t : S]$ of sort $S$ in context $\vbar : J$ in $\Term(\Sigma; \vbar : J)_S^\calR$ that $f_S \circ t^A = t^B \circ f^J : A^J \to B_S$, which is true by \eqref{synt_to_nat_square}. 
\end{proof}

\begin{rmk_sub}
\label{pos_rmk}
Let $\T$ be the relational Horn theory for posets (as in Example \ref{rel_examples}), so that $\T\Mod$ is the cartesian closed category $\Pos$ of posets and monotone functions. When $\calS$ is a singleton, the $\Pos$-categories of algebras for classical single-sorted equational theories with $\T$-relations \eqref{rel_theory} are precisely the \emph{varieties of ordered algebras} considered in \cite{categoricalviewordered} (see \cite[Definition 21 and Remark 22]{categoricalviewordered}) and \cite[\S 5]{Bloom_varieties}; for several examples, see \cite[Example 23]{categoricalviewordered}. In view of Propositions \ref{rel_prop_1} and \ref{rel_prop_2}, our Theorem \ref{variety_thm}, specialized to the case where $\V = \Pos$ and $\calS$ is a singleton, thus recovers \cite[Theorem 41]{categoricalviewordered}. For several examples of classical single-sorted equational theories with $\T$-relations in this context, see \cite[Example 23]{categoricalviewordered}.

Now let $\T$ be the relational Horn theory for (extended) ultrametric spaces (as in Example \ref{rel_examples}), so that $\T\Mod$ is the cartesian closed category $\mathsf{UltMet}$ of (extended) ultrametric spaces. When $\calS$ is a singleton, the $\mathsf{UltMet}$-categories of algebras for classical single-sorted equational theories with $\T$-relations \eqref{rel_theory} are precisely the \emph{varieties of ultra-quantitative algebras} considered in \cite{strongly_finitary_quant} (see \cite[Definition 32]{strongly_finitary_quant}); for some examples, see \cite[Example 33]{strongly_finitary_quant}. In view of Propositions \ref{rel_prop_1} and \ref{rel_prop_2}, our Theorem \ref{variety_thm}, specialized to the case where $\V = \mathsf{UltMet}$ and $\calS$ is a singleton, thus recovers \cite[Theorem 54]{strongly_finitary_quant}. 
\end{rmk_sub}

\begin{rmk_sub}
In future work, we shall study classical $\calS$-sorted equational theories with $\T$-relations \eqref{rel_theory} more thoroughly and in a more general setting, where we shall take $\T$ to be an \emph{arbitrary} relational Horn theory (i.e.~$\T$ need not be reflexive and $\T\Mod$ need not be cartesian closed, \ref{rel_ter_para}). In this future work we shall provide many (more) examples of classical $\calS$-sorted equational theories with $\T$-relations, and we shall also examine the precise relationship between, on the one hand, the classical single-sorted equational theories with $\T$-relations, and, on the other hand, the (single-sorted) \emph{relational algebraic theories} of \cite{Monadsrelational} and the (single-sorted) \emph{quantitative equational theories} of \cite{Quant_alg_reasoning}. 
\end{rmk_sub}

\subsection{Cartesian closed topological, monotopological, and solid categories over $\Set$}
\label{monotop_subsection}

Every cartesian closed category $\V$ whose representable functor $\V(1, -) : \V \to \Set$ is \emph{solid} in the sense of \cite[Definition 25.10]{AHS} is an example of our main setting \eqref{given_data}. This is because $\V$ is then complete and cocomplete \cite[Corollary 25.16]{AHS}, and the functor $\V(1, -) : \V \to \Set$ is faithful \cite[Proposition 25.12]{AHS}. In particular, if $\V(1, -) : \V \to \Set$ is \emph{topological} \cite[Chapter 21]{AHS} or even \emph{monotopological} \cite[Definition 21.38]{AHS}, then $\V(1, -)$ is solid by \cite[Example 25.2 and Proposition 25.11]{AHS}.

Given a relational Horn theory $\T$ \eqref{Horn_theory}, the category $\T\Mod$ is monotopological over $\Set$ by \cite[Proposition 5.5]{Rosickyconcrete}, and so the present class of examples subsumes the previous class of examples considered in \S\ref{rel_subsection}. Prominent examples of cartesian closed categories $\V$ for which $\V(1, -) : \V \to \Set$ is monotopological or even topological (and hence solid) include: the categories $\Top_\calC$ of \emph{$\calC$-generated topological spaces} for a \emph{productive} class $\calC$ of generating spaces considered in \cite{EscardoCCC} (including the category $\mathsf{CGTop}$ of compactly generated spaces); the category $\mathsf{CGHTop}$ of compactly generated Hausdorff spaces; and the categories of quasispaces over (not necessarily \emph{finitary}, \ref{rel_examples}) concrete sites studied in \cite{Dubucquasitopoi}, including the categories of convergence spaces, subsequential spaces, bornological sets, pseudotopological spaces, quasitopological spaces, diffeological spaces \cite{Baezsmooth}, and quasi-Borel spaces \cite{Heunenprobability}. 

When $\V(1, -) : \V \to \Set$ is topological, the $\V$-enriched $\calS$-sorted equational theories of Definition \ref{enriched_theory} coincide with the $\V$-enriched $\calS$-sorted equational theories considered in \cite{Initial_algebras} (where $\V$ is assumed to be topological over $\Set$ but not necessarily cartesian closed); see \cite[Definition 4.1.4 and Theorem 4.1.17]{Initial_algebras}. For many examples of $\V$-enriched $\calS$-sorted equational theories in this setting, see \cite[\S 5]{Initial_algebras}.

\subsection{$\oCPO$ and $\DCPO$}
\label{CPO_subsection}

Recall that an \emph{$\omega$-complete partial order}, or \emph{$\omega$-cpo}, is a partially ordered set in which every $\omega$-chain has a supremum, while a \emph{directed complete partial order}, or \emph{dcpo}, is a partially ordered set in which every directed subset has a supremum. We write $\oCPO$ (resp.~$\DCPO$) for the category whose objects are $\omega$-cpos (resp.~dcpos) and whose morphisms are the \emph{$\omega$-continuous} (resp.~\emph{directed-continuous}) functions, i.e.~the monotone functions that preserve suprema of $\omega$-chains (resp.~directed subsets). The categories $\oCPO$ and $\DCPO$ are well known to be complete, cocomplete, and cartesian closed. The forgetful functor $|-| : \oCPO \to \Set$ that sends an $\omega$-cpo to its underlying set is faithful and represented by the terminal $\omega$-cpo $1$ (the discrete partially ordered set on a singleton set). The case for $\DCPO$ is identical, and so $\oCPO$ and $\DCPO$ are both examples of our main setting \eqref{given_data}. 

We now show, analogously to \S\ref{rel_subsection}, that ($\oCPO$)-enriched $\calS$-sorted equational theories admit an even more concrete, syntactic formulation in terms of \emph{classical $\calS$-sorted equational theories with ($\oCPO$)-relations} \eqref{CPO_theory} (we briefly discuss the case for $\DCPO$ in Remark \ref{dcpo_rmk}). 

\begin{defn_sub}
\label{CPO_relation}
Let $\Sigma$ be a classical $\calS$-sorted signature \eqref{signature}. Given an $\calS$-sorted variable context $\vbar$ and a sort $S \in \calS$, a \textbf{$(\Sigma, \oCPO)$-relation of sort $S$ in context $\vbar$} is a formal expression of the form 
\[ \left[\vbar \vdash \bigvee_{n \geq 0} t_n \doteq t : S\right], \] where $t_n, t \in \Term(\Sigma; \vbar)_S$ ($n \geq 0$). 

A \textbf{$(\Sigma, \oCPO)$-relation in context $\vbar$} is a $(\Sigma, \oCPO)$-relation of sort $S$ in context $\vbar$ for some sort $S \in \calS$, while a \textbf{$(\Sigma, \oCPO)$-relation (in context)} is a $(\Sigma, \oCPO)$-relation in context $\vbar$ for some context $\vbar$. 
\end{defn_sub}

\noindent Given an $\omega$-cpo $X$ and elements $x, y \in |X|$, we sometimes write $X \models x \leq y$ to mean that $x \leq y$ in the partially ordered set $X$. 

\begin{defn_sub}
\label{CPO_relation_sat}
Let $\Sigma$ be a classical $\calS$-sorted signature \eqref{signature}, let $A$ be a $\Sigma$-algebra, let $J \in \N_\calS$, and let $\left[\vbar : J \vdash \bigvee_{n \geq 0} t_n \doteq t : S\right]$ be a $(\Sigma, \oCPO)$-relation in context $\vbar : J$ \eqref{notation_para}. We say that $A$ \textbf{satisfies} the $(\Sigma, \oCPO)$-relation $\left[\vbar : J \vdash \bigvee_{n \geq 0} t_n \doteq t : S\right]$ if 
\[ \oCPO\left(A^J, A_S\right) \models t_n^A \leq t_{n+1}^A \] for all $n \geq 0$ and
\[ \bigvee_{n \geq 0} t_n^A = t^A \] in the $\omega$-cpo $\oCPO\left(A^J, A_S\right)$.      
\end{defn_sub}

\begin{defn_sub}
\label{CPO_theory}
A \textbf{classical $\calS$-sorted equational theory with ($\oCPO$)-relations} is a triple $\calP = (\Sigma, \calR, \calE)$ consisting of a classical $\calS$-sorted signature $\Sigma$ \eqref{signature}, a set $\calR$ of $(\Sigma, \oCPO)$-relations in context \eqref{CPO_relation}, and a set $\calE$ of syntactic $\Sigma$-equations in context \eqref{syntax_para}. A \textbf{$\calP$-algebra} is a $\Sigma$-algebra that satisfies each $(\Sigma, \oCPO)$-relation in $\calR$ \eqref{CPO_relation_sat} and each syntactic $\Sigma$-equation in $\calE$ \eqref{sat_synt}. We write $\calP\Alg \hookrightarrow \Sigma\Alg$ for the full sub-($\oCPO$)-category consisting of the $\calP$-algebras, which is equipped with a faithful ($\oCPO$)-functor $U^\calP : \calP\Alg \to \oCPO^\calS$ obtained by restricting the faithful ($\oCPO$)-functor $U^\Sigma : \Sigma\Alg \to \oCPO^\calS$. When $\calS$ is a singleton, we instead say \textbf{classical single-sorted equational theory with ($\oCPO$)-relations}.  
\end{defn_sub}

\noindent In Propositions \ref{CPO_prop_1} and \ref{CPO_prop_2}, we show that ($\oCPO$)-enriched $\calS$-sorted equational theories have the same ``expressive power'' as classical $\calS$-sorted equational theories with ($\oCPO$)-relations.

\begin{prop_sub}
\label{CPO_prop_1}
For every ($\oCPO$)-enriched $\calS$-sorted equational theory $\calT = (\Sigma, \calE)$, there is a classical $\calS$-sorted equational theory with ($\oCPO$)-relations $\calP_\calT$ such that $\calT\Alg \cong \calP_\calT\Alg$ in $(\oCPO)\CAT/\oCPO^\calS$.
\end{prop_sub}

\begin{proof}
The proof is similar to that of Proposition \ref{rel_prop_1}. The classical $\calS$-sorted signature for $\calP_\calT$ is $|\Sigma|$, the underlying classical $\calS$-sorted signature of $\Sigma$ \eqref{ord_sig}. To define the set $\calR$ of $(|\Sigma|, \oCPO)$-relations for $\calP_\calT$, we consider for each operation symbol $\sigma \in \Sigma$ of type $(J, S, P)$ and each
$\omega$-chain $\left(p_n\right)_{n \geq 0}$ in $P$ the $(|\Sigma|, \oCPO)$-relation 
\begin{equation}
\label{cpo_eq}
\left[\vbar : J \vdash \bigvee_{n \geq 0} \sigma_{p_n}\left(\vbar\right) \doteq \sigma_{\bigvee_n p_n}\left(\vbar\right) : S\right]
\end{equation} 
of sort $S$ in context $\vbar : J$ \eqref{notation_para}. We then define $\calR$ to consist of precisely these $(|\Sigma|, \oCPO)$-relations. We set $\calP_\calT := \left(|\Sigma|, \calR, \calE\right)$, using the set $\calE$ of syntactic $|\Sigma|$-equations from $\calT$. 

We now show that $\calT\Alg \cong \calP_\calT\Alg$ in $(\oCPO)\CAT/\oCPO^\calS$. Given a $\calT$-algebra $A$, we have the underlying $|\Sigma|$-algebra $|A|$ \eqref{ord_sig_para} with the same carrier object $A$, and we show that $|A|$ is a $\calP_\calT$-algebra. The $|\Sigma|$-algebra $|A|$ already satisfies all of the syntactic $|\Sigma|$-equations in $\calE$, because the $\Sigma$-algebra $A$ is a $\calT$-algebra. Now let $\sigma \in \Sigma$ be of type $(J, S, P)$, let $\left(p_n\right)_{n \geq 0}$ be an $\omega$-chain in $P$, and let us show that the $|\Sigma|$-algebra $|A|$ satisfies the $(|\Sigma|, \oCPO)$-relation \eqref{cpo_eq}
of sort $S$ in context $\vbar : J$. So we must show that $\sigma_{p_n}^A \leq \sigma_{p_{n+1}}^A$ for each $n \geq 0$ and that $\bigvee_{n \geq 0} \sigma_{p_n}^A = \sigma_{\bigvee_n p_n}^A$ in $\oCPO\left(A^J, A_S\right)$. Both claims are true because $\sigma^A : P \to \oCPO\left(A^J, A_S\right)$ is a morphism of $\omega$-cpos. Thus, the $|\Sigma|$-algebra $|A|$ is a $\calP_\calT$-algebra. 

Conversely, let $A$ be a $\calP_\calT$-algebra, and let us define a corresponding $\calT$-algebra $A$ with the same carrier object. For each $\sigma \in \Sigma$ of type $(J, S, P)$, we define
\[ \sigma^A : P \to \oCPO\left(A^J, A_S\right) \] 
\[ p \mapsto \sigma_p^A, \] and $\sigma^A$ is $\omega$-continuous because the $|\Sigma|$-algebra $A$ satisfies each $(|\Sigma|, \oCPO)$-relation in $\calR$. This gives the carrier object $A$ the structure of a $\Sigma$-algebra, which is a $\calT$-algebra because the underlying $|\Sigma|$-algebra is the given $\calP_\calT$-algebra $A$.  

Now recall from \ref{ord_sig_para} that we have a fully faithful $(\oCPO)$-functor $|-|^\Sigma : \Sigma\Alg \to |\Sigma|\Alg$ with $U^{|\Sigma|} \circ |-|^\Sigma = U^\Sigma$. The foregoing establishes that this $(\oCPO$)-functor restricts and corestricts to a fully faithful $(\oCPO)$-functor $\calT\Alg \to \calP_\calT\Alg$ that is bijective-on-objects and thus an isomorphism in $(\oCPO)\CAT/\oCPO^\calS$, as desired.      
\end{proof}

\noindent To establish a converse to Proposition \ref{CPO_prop_1}, we shall require the notion of \emph{$\omega$-cpo presentation} from \cite[\S 3]{Synth_top_prob} (adapted from the notion of \emph{dcpo presentation} developed in \cite{pres_dcpos}).

\begin{defn_sub}[\cite{Synth_top_prob, pres_dcpos}]
\label{cpo_pres}
An \textbf{$\omega$-cpo presentation} is a pair $(P, \lhd)$ consisting of a preordered set $P = \left(|P|, \leq\right)$ and a \textbf{cover relation} $\lhd \subseteq |P| \times \mathsf{Pow}(|P|)$ such that for all $(p, U) \in |P| \times \Pow(|P|)$ we have $p \lhd U$ only if $U$ is an $\omega$-chain (when equipped with the preorder inherited from $P$). A \textbf{morphism of $\omega$-cpo presentations} $f : (P, \lhd) \to (P', \lhd')$ is a monotone function $f : P \to P'$ that \emph{preserves covers}, in the sense that $p \lhd U$ implies $f(p) \lhd f[U]$ for all $(p, U) \in |P| \times \Pow(|P|)$.

Every $\omega$-cpo $P$ can be regarded as an $\omega$-cpo presentation $(P, \lhd)$ with $p \lhd U$ iff $U \subseteq |P|$ is an $\omega$-chain and $p \leq \bigvee U$.  
\end{defn_sub} 

\begin{prop_sub}[\cite{Synth_top_prob, pres_dcpos}]
\label{free_cpo_prop}
Let $(P, \lhd)$ be an $\omega$-cpo presentation. Then there is a free $\omega$-cpo over $(P, \lhd)$, i.e.~there is an $\omega$-cpo $P_\omega$ and a morphism $\eta_P : (P, \lhd) \to P_\omega$ of $\omega$-cpo presentations such that for any $\omega$-cpo $X$ and any morphism of $\omega$-cpo presentations $f : (P, \lhd) \to X$, there is a unique morphism $f^\sharp : P_\omega \to X$ of $\omega$-cpos such that $f^\sharp \circ \eta_P = f$. \qed
\end{prop_sub}

\noindent We can now establish the following converse to Proposition \ref{CPO_prop_1}.

\begin{prop_sub}
\label{CPO_prop_2}
For every classical $\calS$-sorted equational theory with ($\oCPO$)-relations $\calP$, there is an ($\oCPO$)-enriched $\calS$-sorted equational theory $\calT_\calP$ such that $\calP\Alg \cong \calT_\calP\Alg$ in $(\oCPO)\CAT/\oCPO^\calS$.
\end{prop_sub}

\begin{proof}
The proof is similar to that of Proposition \ref{rel_prop_2}. Let $\calP = (\Sigma, \calR, \calE)$. For each pair $\left(J, S\right) \in \N_\calS \times \calS$, we consider all the $(\Sigma, \oCPO)$-relations $\left[\vbar : J \vdash \bigvee_{n \geq 0} t_n \doteq t : S\right]$ in $\calR$ of sort $S$ in context $\vbar : J$ \eqref{notation_para}. We write $\Term(\Sigma; \vbar : J)_S^\calR$ for the subset of $\Term(\Sigma; \vbar : J)_S$ consisting of the $\Sigma$-terms of sort $S$ in context $\vbar : J$ that occur in some $(\Sigma, \oCPO)$-relation in $\calR$. We then write $P_{J, S}$ for the set $\Term(\Sigma; \vbar : J)_S^\calR$ equipped with the smallest preorder such that $\left(t_n\right)_{n \geq 0}$ is an $\omega$-chain with upper bound $t$ for each $(\Sigma, \oCPO)$-relation in $\calR$ as above. We then obtain an $\omega$-cpo presentation $\left(P_{J, S}, \lhd\right)$ where $\lhd$ consists of precisely the covers $t \lhd \left\{t_n\right\}_{n \geq 0}$ for all of the $(\Sigma, \oCPO)$-relations in $\calR$ as above. By Proposition \ref{free_cpo_prop}, there is a free $\omega$-cpo $P_{J, S, \omega}$ over $\left(P_{J, S}, \lhd\right)$, equipped with a morphism $\eta_{J, S} : \left(P_{J, S}, \lhd\right) \to P_{J, S, \omega}$ of $\omega$-cpo presentations. We then augment the classical $\calS$-sorted signature $\Sigma$ to obtain an ($\oCPO$)-enriched $\calS$-sorted signature $\Sigma^+$ by adjoining, for each pair $\left(J, S\right) \in \N_\calS \times \calS$, an operation symbol $\sigma_{J, S}$ of type $\left(J, S, P_{J, S, \omega}\right)$. We then augment the set $\calE$ of syntactic $\Sigma$-equations to obtain a set $\calE^+$ of syntactic $\left|\Sigma^+\right|$-equations by adjoining, for each pair $\left(J, S\right) \in \N_\calS \times \calS$, each $(\Sigma, \oCPO)$-relation $\left[\vbar : J \vdash \bigvee_{n \geq 0} t_n \doteq t : S\right]$ in $\calR$ of sort $S$ in context $\vbar : J$, and each $n \geq 0$, the following syntactic $\left|\Sigma^+\right|$-equation of sort $S$ in context $\vbar : J$:
\begin{equation}
\label{cpo_eq_2}
\left[\vbar : J \vdash \left(\sigma_{J, S}\right)_{\eta_{J, S}\left(t_n\right)}\left(\vbar\right) \doteq t_n : S\right],
\end{equation} 
as well as the following syntactic $\left|\Sigma^+\right|$-equation of sort $S$ in context $\vbar : J$:
\begin{equation}
\label{cpo_eq_3}
\left[\vbar : J \vdash \left(\sigma_{J, S}\right)_{\eta_{J, S}(t)}\left(\vbar\right) \doteq t : S\right].
\end{equation} 
We then define $\calT_\calP := \left(\Sigma^+, \calE^+\right)$. Note that for each $\calT_\calP$-algebra $A$ and each $(J, S) \in \N_\calS \times \calS$, the morphism of $\omega$-cpo presentations
\[ \sigma_{J, S}^A \circ \eta_{J, S} : \left(P_{J, S}, \lhd\right) \to \oCPO\left(A^J, A_S\right) \] is given by the rule $\left[\vec{v} : J \vdash t : S\right] \mapsto t^A$ for each $\left[\vec{v} : J \vdash t : S\right] \in \Term(\Sigma; \vbar : J)_S^\calR$, because the $\calT_\calP$-algebra $A$ satisfies the syntactic equations \eqref{cpo_eq_2} and \eqref{cpo_eq_3}. 

We now show that $\calP\Alg \cong \calT_\calP\Alg$ in $(\oCPO)\CAT/\oCPO^\calS$. Given a $\calP$-algebra $A$, we define a corresponding $\calT_\calP$-algebra $A^+$ with the same carrier object. For each $\sigma \in \Sigma \subseteq \Sigma^+$, we set $\sigma^{A^+} := \sigma^A$. For each pair $\left(J, S\right) \in \N_\calS \times \calS$, we define
\[ \sigma_{J, S}^{A^+} : P_{J, S, \omega} \to \oCPO\left(A^J, A_S\right) \] to be the unique $\omega$-continuous function that is induced by the morphism of $\omega$-cpo presentations \[ f_{J, S}^A : \left(P_{J, S}, \lhd\right) \to \oCPO\left(A^J, A_S\right) \] 
given by
\[ [\vbar : J \vdash t : S] \mapsto t^A. \] That $f^A_{J, S}$ is actually a morphism of $\omega$-cpo presentations readily follows from the fact that the $\calP$-algebra $A$ satisfies each $(\Sigma, \oCPO)$-relation in $\calR$. This equips the carrier object $A$ with the structure of a $\Sigma^+$-algebra $A^+$. As in the proof of Proposition \ref{rel_prop_2}, it readily follows that the $\Sigma^+$-algebra $A^+$ is a $\calT_\calP$-algebra. 

Conversely, let $A$ be a $\calT_\calP$-algebra, and let us define a corresponding $\calP$-algebra $A^-$ with the same carrier object. Since $\Sigma \subseteq \Sigma^+$, the carrier object $A$ already carries the structure of a $\Sigma$-algebra $A^-$. Furthermore, since $\calE \subseteq \calE^+$, the $\Sigma$-algebra $A^-$ satisfies each syntactic $\Sigma$-equation in $\calE$. So it remains to show that the $\Sigma$-algebra $A^-$ satisfies each $(\Sigma, \oCPO)$-relation of $\calR$. For each pair $\left(J, S\right) \in \N_\calS \times \calS$, let $\left[\vbar : J \vdash \bigvee_{n \geq 0} t_n \doteq t : S\right]$ be a $(\Sigma, \oCPO)$-relation in $\calR$ of sort $S$ in context $\vbar : J$. To show that the $\Sigma$-algebra $A^-$ satisfies this $(\Sigma, \oCPO)$-relation, we must show that
\[ \oCPO\left(A^J, A_S\right) \models t_n^A \leq t_{n+1}^A \] for each $n \geq 0$ and that $\bigvee_{n \geq 0} t_n^A = t^A$. Since $A$ is a $\calT_\calP$-algebra, we have the $\omega$-continuous function
\[ \sigma_{J, S}^A : P_{J, S, \omega} \to \oCPO\left(A^J, A_S\right). \] Since $P_{J, S} \models t_n \leq t_{n+1}$ and $P_{J, S} \models t_n \leq t$ for each $n \geq 0$, we deduce that
\[ \oCPO\left(A^J, A_S\right) \models \sigma_{J, S}^A\left(\eta_{J, S}\left(t_n\right)\right) \leq \sigma_{J, S}^A\left(\eta_{J, S}\left(t_{n+1}\right)\right) \] and
\[ \oCPO\left(A^J, A_S\right) \models \sigma_{J, S}^A\left(\eta_{J, S}\left(t_n\right)\right) \leq \sigma_{J, S}^A\left(\eta_{J, S}(t)\right) \] 
for each $n \geq 0$. Furthermore, since $t \lhd \left\{t_n\right\}_{n \geq 0}$ in $\left(P_{J, S}, \lhd\right)$, we also deduce that $\eta_{J, S}(t) \leq \bigvee_{n \geq 0} \eta_{J, S}\left(t_n\right)$ in $P_{J, S, \omega}$, and hence that \[\sigma_{J, S}^A\left(\eta_{J, S}(t)\right) \leq \sigma_{J, S}^A\left(\bigvee_{n \geq 0} \eta_{J, S}\left(t_n\right)\right) = \bigvee_{n \geq 0} \sigma_{J, S}^A\left(\eta_{J, S}\left(t_n\right)\right) \] in $\oCPO\left(A^J, A_S\right)$. Since $A$ is a $\calT_\calP$-algebra, we have $\sigma_{J, S}^A\left(\eta_{J, S}\left(t_n\right)\right) = t_n^A$ for each $n \geq 0$ as well as $\sigma_{J, S}^A\left(\eta_{J, S}(t)\right) = t^A$, from which we deduce that $t_n^A \leq t_{n+1}^A$ and $t_n^A \leq t^A$ for each $n \geq 0$ and that 
\[ t^A \leq \bigvee_{n \geq 0} t_n^A, \] so that $t^A = \bigvee_{n \geq 0} t_n^A$, as desired. Exactly as in the proof of Proposition \ref{rel_prop_2}, we have thus established a bijective correspondence between $\calP$-algebras and $\calT_\calP$-algebras, which extends to the desired isomorphism $\calT_\calP\Alg \cong \calP\Alg$ in $(\oCPO)\CAT/\oCPO^\calS$.  
\end{proof}

\begin{rmk_sub}
\label{ocpo_rmk}
When $\calS$ is a singleton, the $(\oCPO)$-categories of algebras for classical single-sorted equational theories with ($\oCPO$)-relations are precisely the \emph{varieties of continuous algebras} considered in \cite{strongly_finitary_cts} (see \cite[Definition 4.11]{strongly_finitary_cts}); for some examples, see \cite[Example 4.12]{strongly_finitary_cts}. In view of Propositions \ref{CPO_prop_1} and \ref{CPO_prop_2}, our Theorem \ref{variety_thm}, specialized to the case where $\V = \oCPO$ and $\calS$ is a singleton, thus recovers \cite[Corollary 5.9]{strongly_finitary_cts}. 
\end{rmk_sub}

\begin{rmk_sub}
\label{dcpo_rmk}
Using the notion of \emph{dcpo presentation} from \cite{pres_dcpos}, one could (in principle) prove results analogous to Propositions \ref{CPO_prop_1} and \ref{CPO_prop_2} for $\DCPO$-enriched $\calS$-sorted equational theories, allowing us to recover \cite[Corollary 6.16]{strongly_finitary_cts} as an instance of our Theorem \ref{variety_thm} (with $\V = \DCPO$ and $\calS$ a singleton); but we do not pursue this in detail here.
\end{rmk_sub}

\bibliographystyle{amsalpha}
\bibliography{mybib}

\end{document}